\documentclass[11pt,english]{amsart}
\usepackage[T1]{fontenc}
\usepackage[utf8]{inputenc}
\usepackage[a4paper]{geometry}
\usepackage[active]{srcltx}
\usepackage{xcolor}
\usepackage{verbatim}
\usepackage{amstext}
\usepackage{amsthm}
\usepackage{amssymb}
\usepackage{graphicx}
\usepackage{setspace}
\usepackage{esint}
\usepackage[all]{xy}
\PassOptionsToPackage{normalem}{ulem}
\usepackage{ulem}
\makeatletter

\providecolor{lyxadded}{rgb}{0,0,1}
\providecolor{lyxdeleted}{rgb}{1,0,0}

\numberwithin{equation}{section}
\numberwithin{figure}{section}
  \theoremstyle{plain}
  \newtheorem*{thm*}{\protect\theoremname}
\theoremstyle{plain}
\newtheorem{thm}{\protect\theoremname}[section]
  \theoremstyle{definition}
  \newtheorem{defn}[thm]{\protect\definitionname}
  \theoremstyle{plain}
  \newtheorem{lem}[thm]{\protect\lemmaname}
  \theoremstyle{remark}
  \newtheorem{rem}[thm]{\protect\remarkname}
  \theoremstyle{plain}
  \newtheorem{prop}[thm]{\protect\propositionname}
  \theoremstyle{definition}
  \newtheorem*{example*}{\protect\examplename}
  \theoremstyle{plain}
  \newtheorem{cor}[thm]{\protect\corollaryname}
  \theoremstyle{definition}
  \newtheorem{example}[thm]{\protect\examplename}
  \theoremstyle{remark}
  \newtheorem*{rem*}{\protect\remarkname}

\usepackage{amsfonts}

\usepackage[all]{xy}
\usepackage{amscd}
\usepackage{pstricks}
\usepackage{pst-plot}

\renewcommand{\hom}{\mathrm{Hom}}

\newcommand{\Hom}{\mathop{\mathrm{Hom}}\nolimits}

\newcommand{\Ad}{\mathop{\mathrm{Ad}}\nolimits}

\def\lie{\mathsf{Lie}}
\def\quot{/\!\!/}

\title[Principal Schottky Bundles]{Principal Schottky Bundles over Riemann surfaces}

\author[A. C. Casimiro]{A. C. Casimiro}

\address{Departamento Matem\'{a}tica and Centro de  Matem\'{a}tica e Aplica\c{c}\~{o}es, Faculdade de Ci\^{e}ncias e Tecnologia,
Universidade Nova de Lisboa, Quinta da Torre, 2829-516 Caparica, Portugal}

\email{amc@fct.unl.pt}

\author[S. Ferreira]{S. Ferreira}

\address{Escola Superior de Tecnologia e Gest\~{a}o de Leiria,
Campus 2, Morro do Lena, Alto do Vieiro,
Apartado 4163, 2411-901 Leiria, Portugal }

\email{susfer@ipleiria.pt}

\author[C. Florentino]{C. Florentino}

\address{Departamento de Matem\'{a}tica, Faculdade de Ci\^{e}ncias, Univ. de Lisboa,  Edf. C6, Campo Grande 1749-016 Lisboa, Portugal}

\email{caflorentino@fc.ul.pt}

\thanks{This work was partially supported by the projects  PTDC/MAT/120411/2010, PTDC/MAT-GEO/0675/2012 and EXCL/MAT-GEO/0222/2012, FCT, Portugal, and
by the USA NSF grants DMS 1107452, 1107263, 1107367 "RNMS: GEometric structures And Representation varieties" (the GEAR Network)}

\keywords{Representations of the fundamental group, character varieties, principal bundles, moduli spaces, Riemann surfaces, Schottky bundles, uniformization}

\usepackage{babel}
\providecommand{\corollaryname}{Corollary}
  \providecommand{\definitionname}{Definition}
  \providecommand{\examplename}{Example}
  \providecommand{\lemmaname}{Lemma}
  \providecommand{\propositionname}{Proposition}
  \providecommand{\remarkname}{Remark}
  \providecommand{\theoremname}{Theorem}
\providecommand{\theoremname}{Theorem}

\makeatother

\usepackage{babel}
  \providecommand{\corollaryname}{Corollary}
  \providecommand{\definitionname}{Definition}
  \providecommand{\examplename}{Example}
  \providecommand{\lemmaname}{Lemma}
  \providecommand{\propositionname}{Proposition}
  \providecommand{\remarkname}{Remark}
  \providecommand{\theoremname}{Theorem}
\providecommand{\theoremname}{Theorem}

\begin{document}
\begin{abstract}
We introduce and study (strict) Schottky $G$-bundles over a compact
Riemann surface $X$, where $G$ is a connected reductive algebraic
group. Strict Schottky representations are shown to be related to
branes in the moduli space of $G$-Higgs bundles over $X$, and we
prove that all Schottky $G$-bundles have trivial topological type.
Generalizing the Schottky moduli map introduced in \cite{Florentino}
to the setting of principal bundles, we prove its local surjectivity
at the good and unitary locus. Finally, we prove that the Schottky
map is surjective onto the space of flat bundles for two special classes:
when $G$ is an abelian group over an arbitrary $X$, and the case
of a general $G$-bundle over an elliptic curve.
\end{abstract}

\maketitle

\section{Introduction and Main Results}

\subsection{Schottky uniformizations}

The classical Fuchsian uniformization theorem provides an explicit
parameterization of all Riemann surfaces $X$ of genus $g\geq2$:
every such $X$ can be obtained as $\mathbb{H}/\Gamma$, a quotient
of the upper half-plane $\mathbb{H}$ by a Fuchsian group $\Gamma\subset PSL_{2}\mathbb{R}$,
isomorphic to the fundamental group of $X$, $\pi_{1}(X)$. A less
well-known result, the so-called ``retrosection theorem'', or Schottky
uniformization, asserts that we can also write $X\cong\Omega/\Sigma$,
for a certain \emph{free group} of Möbius transformations $\Sigma\subset PSL_{2}\mathbb{C}$
of rank $g$ (called, in this context, a \emph{Schottky group}) and
region of discontinuity (for the $\Sigma$-action) $\Omega\subset\mathbb{CP}^{1}$
(see \cite{Bers75,Ford}).

These are two very different parametrizations: the Fuchsian one is
essentially unique, and provides an identification between Teichmüller
space and one component of the (real) \emph{character variety} $\hom(\pi_{1}(X),PSL_{2}\mathbb{R})/PSL_{2}\mathbb{R}$
(the quotient of representations $\pi_{1}(X)\to PSL_{2}\mathbb{R}$
by conjugation); by contrast, the Schottky one is defined on a less
explicit subset of $\hom(\Sigma,PSL_{2}\mathbb{C})/PSL_{2}\mathbb{C}$,
having the advantage of providing manifestly holomorphic coordinates.

Passing from surfaces to holomorphic bundles over a \emph{fixed} Riemann
surface $X$, it is natural to consider analogous explicit parametrizations.
In their famous papers \cite{NarasimhanSeshadri2,NarasimhanSeshadri1},
Narasimhan and Seshadri proved that every polystable vector bundle
over $X$, of degree zero, can be obtained from a (unique up to conjugation)
unitary representation. Ramanathan generalised Narasimhan-Seshadri's
results to principal $G$-bundles, where $G$ is any reductive algebraic
group over $\mathbb{C}$ (see \cite{Ramanathan1,Ramanathan2}).

More precisely, a representation $\rho:\pi_{1}(X)\to K\subset G$
into $K$, a maximal compact subgroup of $G$, defines a holomorphic
$G$-bundle over $X=\mathbb{H}/\pi_{1}(X)$, equipped with a natural
flat connection: 
\begin{equation}
E_{\rho}\,:=\,(\mathbb{H}\times G)/_{\!\!\rho}\ \pi_{1}(X),\label{eq:Shot-G-bdl}
\end{equation}
using the diagonal action of $\pi_{1}(X)$, via $\rho$, on the trivial
$G$-bundle $G\times\mathbb{H}\to\mathbb{H}$. A special case of the
results of Narasimhan, Seshadri and Ramanathan is that a holomorphic
$G$-bundle over $X$, which admits a flat connection, is polystable
if and only if it can be written in the above form, for some $\rho:\pi_{1}(X)\to K$,
unique up to conjugation. Their result can thus be seen as a bundle
version of classical Fuchsian uniformization, and identifies the moduli
space of flat semistable $G$-bundles with the ``real character variety''
\[
\hom(\pi_{1}(X),K)/K.
\]

The question of whether some sort of Schottky uniformization can be
obtained for a large class of holomorphic $G$-bundles is still an
open problem, as far as we know.\footnote{Interestingly, the consideration of the Schottky uniformization problem
for vector bundles over Mumford curves, in the framework of $p$-adic
analysis, has furnished stronger results. (see \cite{Faltings83}).} Florentino studied the case of \emph{vector} bundles and obtained
some partial results (\cite{Florentino}), showing that all flat line
bundles, and all flat vector bundles over an elliptic curve are Schottky
bundles: these can be defined as in \eqref{eq:Shot-G-bdl} for certain
representations $\rho$ of a \emph{free group of rank $g$} into the
general linear group $GL_{n}\mathbb{C}$. Moreover, an open subset
of the moduli space of degree zero semistable vector bundles consists
of Schottky vector bundles. This study was motivated by an attempt
to develop an analytic theory of non-abelian theta functions and their
relation to the spaces of conformal blocks in conformal field theory
(see \cite{Beauville,FlorentinoMouraoNunes03,Tyu}).

Schottky (\emph{principal}) $G$-bundles were defined by Florentino
and Ludsteck, for a general complex reductive algebraic group $G$
(\cite{FlorentinoLudsteck}). They showed that there exists a natural
equivalence between the categories of unipotent representations of
a Schottky group of rank $g$ and unipotent holomorphic vector bundles
over Riemann surface of genus $g$.

In this paper, we generalize the results of the article \cite{Florentino}
in two different ways: we replace $GL_{n}\mathbb{C}$ by an arbitrary
connected complex reductive group $G$, and we consider a more general
definition of Shottky representations, allowing all marked generators
to be represented in the center of $G$.

\subsection{Main results}

We now summarize our main results, emphasizing the novelties in the
principal bundle case, while describing the contents of each section.
Consider the usual presentation 
\begin{equation}
\pi_{1}(X)=\left\langle \alpha_{1},\cdots,\alpha_{g},\beta_{1},\cdots,\beta_{g}\,|\:{\textstyle \prod_{i=1}^{g}}\alpha_{i}\beta_{i}\alpha_{i}^{-1}\beta_{i}^{-1}=1\right\rangle ,\label{eq:pi1}
\end{equation}
of the fundamental group of a fixed Riemann surface $X$, of genus
$g\geq1$ (we are implicitly choosing a base point $x_{0}\in X$,
but this is irrelevant when considering isomorphism classes of representations).
A representation $\rho:\pi_{1}(X)\to G$ is said to be \emph{Schottky}
(with respect to our choice of generators above) if $\rho(\alpha_{i})$
is in the \emph{center} $Z=Z_{G}$ of $G$ for all $i=1,\cdots,g$.
These include what we call \emph{strict Schottky representations},
which verify $\rho(\alpha_{i})=e$ for all $i=1,\cdots,g$, with $e$
the identity of $G$. Although the definitions require a choice of
generators for $\pi_{1}(X)$, our results are independent of such
choices. Thus, from an algebro-geometric perspective, Schottky representations
(up to conjugation) are naturally parametrized by the affine \emph{geometric
invariant theory} (GIT) quotient 
\[
\mathbb{S}:=\Hom\left(F_{g},\,Z\times G\right)\quot G,
\]
where $F_{g}$ denotes a fixed free group of rank $g$ (see Proposition
\ref{prop:SchottkyAndFreeGroup}). Besides these definitions and first
properties, in Section \ref{sec:Schottky-Representations} we describe
the irreducible components of the Schottky space $\mathbb{S}$ and
prove the existence of good and unitary Schottky representations for
$g\geq2$.

Strict Schottky representations have the following natural topological
interpretation. Suppose that $M$ is a 3-manifold whose boundary is
$X$, and the natural morphism $i_{*}:\pi_{1}(X)\to\pi_{1}(M)$ induced
by the inclusion $i:X\hookrightarrow M$, has all the $\alpha_{i}$
in its kernel and the $\beta_{i}$ are free, $i=1,\cdots,g$. Then
it is easy to see that \emph{strict} Schottky representations are
the representations of $\pi_{1}(X)$ which ``extend to $M$'', meaning
that they factor through $i_{*}$ (note that $\pi_{1}(M)$ is indeed
a free group of rank $g$). In addition to its relation to the uniformization
problems for holomorphic $G$-bundles, Schottky representations also
appear in a different context, related to non-abelian Hodge theory:
recently, Baraglia and Schaposnik considered $G$-Higgs bundles over
a Riemann surface equipped with an anti-holomorphic involution and
showed that, inside the moduli space of $G$-Higgs bundles, the locus
of those which are fixed by an associated involution define what is
called an $(A,B,A)$-brane (\cite{BaragliaSchaposnick}). In Section
\ref{sec:Higgs-bundles,-Schottky,Branes}, we identify all strict
Schottky representations as elements of this brane (see \cite[Proposition 43]{BaragliaSchaposnick}
and Proposition \ref{prop:Schottky-brane}). The study of branes is
of great interest in connection with mirror symmetry and the geometric
Langlands correspondence (see \cite{KapustinWitten}).

Section \ref{sec:Schottky-G-bundles} provides the definition of Schottky
$G$-bundles and their relation to Schottky vector bundles in terms
of associated bundles. A Schottky (principal) $G$-bundle over $X$
is defined to be a holomorphic bundle which is isomorphic to a bundle
of the form \eqref{eq:Shot-G-bdl}, for some Schottky representation
$\rho$ (so that its conjugation class $[\rho]$ belongs to $\mathbb{S}$).
Similarly, we define strict Schottky bundles. Note that all Schottky
bundles, being defined by representations of $\pi_{1}(X)$, necessarily
admit a flat holomorphic connection.

The association of a Schottky $G$-bundle to a Schottky representation
defines what we call the \emph{Schottky uniformization map}: 
\[
\mathbf{W}\,:\,\mathbb{S}\to M_{G},
\]
where $M_{G}$ stands for the set of isomorphism classes of $G$-bundles
over $X$ admitting a flat connection.%
{} Two important properties of $\mathbf{W}$ are in clear contrast with
the Narasimhan-Seshadri-Ramanathan uniformization (see Remark \ref{rem:nosemistable}(1)):

(1) A (strict) Schottky bundle is not necessarily semistable (contrary
to those coming from unitary representations $\rho:\pi_{1}(X)\to K$);

(2) If $E=E_{\rho}$ is a Schottky bundle, then $[\rho]\in\mathbb{S}$
is not unique in general, and the preimage $\mathbf{W}^{-1}([E])$
is typically infinite.

By results of Ramanathan \cite{Ramanathan1}, further developed in
\cite{Li}, the topological invariants of flat $G$-bundles are labeled
by elements in $\pi_{1}(DG)$, where $DG$ is the derived group of
$G$. Moreover, there exist flat $G$-bundles with all possible topological
types. In Section \ref{sec:Topological-type}, we prove that the Schottky
case is particularly simple (Theorem \ref{thm:Schottky bdl has trivial type}):
\begin{thm*}
\textbf{\emph{(A)}} Every Schottky $G$-bundle is topologically trivial. 
\end{thm*}
In Section \ref{sec:Uniformization-map} we define and study the notion
of \emph{analytic equivalence} of representations and consider the
period map, for later use in computing the derivative of the Schottky
map. In general, Schottky representations and strict ones are distinct.
Analytic equivalence allows to prove that, for Schottky bundles, the
distinction between the strict and the general case is not relevant
when $G$ has a \emph{connected} center (Proposition \ref{connected-center}).

In Section \ref{sec:Schottky-moduli-map}, we consider the tangent
spaces to Schottky space, describe them in terms of the first cohomology
group of $F_{g}$ in certain $F_{g}$-modules, and compute the dimension
of the Schottky space $\mathbb{S}:=\mathcal{S}\quot G$. We characterize
the kernel of the derivative of the Schottky moduli map at a good
Schottky representation. We also prove that the good locus of strict
Schottky space is a Lagrangian submanifold of the complex manifold
of the smooth points of $\hom(\pi_{1}(X),G)\quot G$.

Let $\mathcal{M}_{G}$ denote the moduli space of semistable $G$-bundles
over $X$ and consider the restricted map called the \emph{Schottky
moduli map} 
\[
\mathbf{V}\,:\,\mathbb{S}^{*}\to\mathcal{M}_{G},
\]
where $\mathbb{S}^{*}:=\mathbf{W}^{-1}\left(\mathcal{M}_{G}\cap M_{G}\right)$
is a dense subset of $\mathbb{S}$. With their natural complex structures,
this gives now a holomorphic map between the smooth locus of the corresponding
spaces. In Section \ref{sec:Surjectivity-of-the} we compute the derivative
of the Schottky moduli map at a good and unitary representation (assuming
also that $[E_{\rho}]$ is a smooth point of $\mathcal{M}_{G}$),
proving that it is an isomorphism when $G$ is semisimple (Corollary
\ref{cor:maintheorem}). In the more general case of reductive $G$,
the Schottky moduli map will be a submersion (Theorem \ref{thm: main thm}). 
\begin{thm*}
\textbf{\emph{(B)}} Let $\rho:\pi_{1}(X)\to G$ be a good and unitary
Schottky representation, such that $[E_{\rho}]$ is a smooth point
in $\mathcal{M}_{G}$. Then, the derivative of the Schottky moduli
map at $[\rho]\in\mathbb{S}^{*}$ has maximal rank. In particular,
locally around $[\rho]$, the Schottky moduli map $\mathbf{V}:\mathbb{S}^{*}\to\mathcal{M}_{G}$
is a submersion, and $\dim\mathbf{V}^{-1}\left(\left[E_{\rho}\right]\right)=g\,\dim Z$. 
\end{thm*}
Finally, in Section \ref{chap:Particular-cases}, we consider two
special classes of Schottky principal bundles: the first case are
$G$-bundles where $G=(\mathbb{C}^{*})^{m}$, for some $m\in\mathbb{N}$,
over a general surface $X$. In this case, since our definition is
more general than the one in \cite{Florentino}, the strict Schottky
condition turns out to be equivalent to flatness (Proposition \ref{prop: C star}).
The second special class consists of Schottky $G$-bundles over a
compact Riemann surface of genus $g=1$, which needs a distinct treatment
than the case $g\geq2$ (Theorem \ref{flat then schottky principal}).
Again, in this case, the Schottky condition is equivalent to flatness. 
\begin{thm*}
\textbf{\emph{(C)}} Let $X$ be an elliptic curve and $E$ a $G$-bundle
over $X$. Then $E$ is Schottky if and only if it admits a flat connection. 
\end{thm*}

\section*{Acknowledgments}

We thank I. Biswas, E. Franco, P. B. Gothen, C. Meneses-Torres and
A. Oliveira for several useful discussions on Schottky bundles and
related subjects, and the referees for clarifying comments. The last
author thanks the organizers of the Simons Center for Geometry and
Physics workshop on Higgs bundles, and L. Schaposnik and D. Baraglia
for details on their construction of (A,B,A) branes.

\section{Schottky Representations\label{sec:Schottky-Representations}}

Given a compact Riemann surface $X$ of genus $g\geq2$, the classical
Schottky uniformization theorem (see \cite{Ford,Bers75}) states that
$X$ is isomorphic to a quotient $\left.\Omega_{\Sigma}\right/\Sigma$,
where $\Sigma\subset PSL_{2}\mathbb{C}$ is a Schottky group and $\Omega_{\Sigma}\subset\mathbb{C}\mathbb{P}^{1}$
is the corresponding region of discontinuity in the Riemann sphere.
Schottky groups are finitely generated free purely loxodromic subgroups
of the Möbius group $PSL_{2}\mathbb{C}$ (see also \cite{Ms}), and
so, $\Sigma$ is the image of a free group $F_{g}$, of $g$ generators,
under a homomorphism $\rho:F_{g}\to PSL_{2}\mathbb{C}$. Naturally,
conjugate homomorphisms define isomorphic surfaces.

In this section, we consider the space of isomorphism classes of representations
of $F_{g}$ into a general complex reductive algebraic group $G$,
and prove some properties of the corresponding algebraic variety.
This is an extension of the notion of Schottky representations studied
in \cite{Florentino}, which were associated to representations of
$F_{g}$ into $GL_{n}\mathbb{C}$.

We begin by fixing some notation. Denote by $\pi_{1}=\pi_{1}(X)$
the fundamental group of $X$, and fix generators $\alpha_{i}$, $\beta_{i}$,
$i=1,\cdots,g$, of $\pi_{1}$ giving the usual presentation 
\begin{equation}
\pi_{1}=\left\langle \alpha_{1},\cdots,\alpha_{g},\beta_{1},\cdots,\beta_{g}\,|\ {\textstyle \prod_{i=1}^{g}}\alpha_{i}\beta_{i}\alpha_{i}^{-1}\beta_{i}^{-1}=1\right\rangle .\label{eq:fund-group}
\end{equation}
Let $G$ be a complex connected reductive algebraic group and denote
by $F_{g}$ a fixed free group of rank $g$, with $g$ fixed generators
$\gamma{}_{1},\cdots,\gamma_{g}$. Since $G$ is algebraic, and $\pi_{1}$
and $F_{g}$ are finitely presented, both $\hom(\pi_{1},G)$ and $\hom(F_{g},G)$
are affine algebraic varieties.

The reductive group $G$ acts by conjugation on $\hom(\pi_{1},G)$
and hence, one can define a geometric invariant theory (GIT) quotient,
called the $G$-character variety of $\pi_{1}$ (also called the \emph{Betti
space} in the context of the non-abelian Hodge theory, see \cite{Simpson}),
as 
\begin{equation}
\mathbb{B}:=\hom(\pi_{1},G)\quot G.\label{eq:Betti-space}
\end{equation}
This is a categorical quotient which, as an affine algebraic variety,
is the maximal spectrum of the $\mathbb{C}$-algebra of $G$-invariant
regular functions in $\mathbb{C}[\hom(\pi_{1},G)]$ (see, for example
\cite[Theorem 3.5]{Newstead}).

\subsection{Schottky representations\label{sec:Schottky-representations}}

Denote by $e\in G$, the unit element of $G$, and by $Z=Z_{G}$ the
center of $G$. 
\begin{defn}
\label{Schottky rep general} A representation $\rho:\pi_{1}\rightarrow G$
is called:
\begin{enumerate}
\item a \emph{Schottky representation} if $\rho(\alpha_{i})\in Z$ for all
$i\in\{1,\cdots,g\}$ 
\item a\textbf{ }\emph{strict Schottky representation}\textbf{ }if $\rho(\alpha_{i})=e$
for all $i\in\{1,\cdots,g\}$ 
\end{enumerate}
The set of Schottky representations is denoted by $\mathcal{S}\subset\hom(\pi_{1},G)$
and the strict ones by $\mathcal{S}_{s}$. Of course, $\mathcal{S}_{s}\subset\mathcal{S}$
and they coincide when $Z=\{e\}$ (i.e., for adjoint groups). 
\end{defn}
For a useful alternative characterization of Schottky representations,
consider the natural short exact sequence of groups 
\[
1\to\ker\varphi\hookrightarrow\pi_{1}\stackrel{\varphi}{\to}F_{g}\to1
\]
where $\varphi$ is the natural epimorphism given, in terms of the
generators, by 
\[
\varphi(\alpha_{i})=e,\quad\mbox{and }\varphi(\beta_{i})=\gamma_{i},\quad\forall i=1,\cdots,g,
\]
so that $\ker\varphi$ is the normal subgroup of $\pi_{1}$ generated
by all $\alpha_{i}$. Schottky representations can also be defined
with respect to the map $\varphi$ as in the following lemma, whose
proof is straightforward (see also \cite{FlorentinoLudsteck}). 
\begin{lem}
\label{lem:alternative-def}Let $\rho\in\hom(\pi_{1},G)$ and let
$\varphi:\pi_{1}\to F_{g}$ be as above. Then
\begin{enumerate}
\item $\rho$ is a Schottky representation if and only if $\rho(\ker\varphi)\subset Z$; 
\item $\rho$ is a\textbf{ }strict Schottky representation\textbf{ }if and
only if $\rho(\ker\varphi)=\{e\}$. 
\end{enumerate}
\end{lem}
Using our fixed generators, we can see $\mathcal{S}$ as an algebraic
subvariety of $\hom(\pi_{1},G)$, isomorphic to $(Z\times G)^{g}$
(and $\mathcal{S}_{s}$ as a smooth subvariety of $\mathcal{S}$,
isomorphic to $G^{g}$). A Schottky representation $\rho\in\mathcal{S}\subset\hom(\pi_{1},G)$
may also be viewed as a representation 
\[
\rho_{F}=(\rho_{1},\rho_{2}):F_{g}\to Z\times G.
\]
Indeed, given $\rho\in\mathcal{S}$, define $\rho_{1}:F_{g}\to Z$,
and $\rho_{2}:F_{g}\to G$ by 
\begin{equation}
\rho_{F}(\gamma_{i})=(\rho_{1}(\gamma_{i}),\rho_{2}(\gamma_{i})):=\left(\rho(\alpha_{i}),\,\rho(\beta_{i})\right)\in Z\times G,\quad\quad i=1,\cdots,g.\label{eq:schottkyrep}
\end{equation}
Conversely, given $\rho_{F}=(\rho_{1},\rho_{2}):F_{g}\to Z\times G$,
we obtain a Schottky representation $\rho\in\mathcal{S}\subset\hom(\pi_{1},G)$
defined by setting $\rho(\alpha_{i}):=\rho_{1}(\gamma_{i})$ and $\rho(\beta_{i}):=\rho_{2}(\gamma_{i})$,
$i=1,\cdots,g$. It is clear that this defines an inclusion of algebraic
varieties 
\begin{equation}
\psi:\mathcal{\hom}(F_{g},Z\times G)\hookrightarrow\hom(\pi_{1},G),\label{eq:inclusion}
\end{equation}
identifying $\hom(F_{g},Z\times G)$ with its image, which is precisely
$\mathcal{S}$. The strict Schottky locus $\mathcal{S}_{s}$ is then
identified with $\hom(F_{g},\{e\}\times G)\simeq\hom(F_{g},G)\simeq G^{g}$,
where the last isomorphism is the evaluation map: $(\sigma:F_{g}\to G)\mapsto\left(\sigma(\gamma_{1}),\cdots,\sigma(\gamma_{g})\right).$
\begin{rem}
\label{rem:Schottky-reps} Our identifications depend on the choice
of generators for $\pi_{1}$ and $F_{g}$, but the algebraic structure
is independent of those choices (different choices provide isomorphic
varieties), as can be easily seen. 
\end{rem}
It is clear that the conjugation action of the reductive group $G$
on $\hom(\pi_{1},G)$ restricts to an action on $\mathcal{S}$ and
on $\mathcal{S}_{s}$. In terms of the identification $\mathcal{S}\cong\hom(F_{g},Z\times G)$,
each element $g\in G$ acts as follows: 
\begin{equation}
(g\cdot\rho_{F})(\gamma)=\left(\rho_{1}(\gamma),\ g\,\rho_{2}(\gamma)\,g^{-1}\right)\quad\mbox{for all }\gamma\in F_{g},\label{eq:accao}
\end{equation}
where $\rho_{F}=(\rho_{1},\rho_{2})$ as above. As before, there exist
a GIT quotient 
\[
\mathbb{S}:=\mathcal{S}\quot G\cong\Hom\left(F_{g},Z\times G\right)\quot G,
\]
which we call the \emph{Schottky space} (in particular, it is a character
variety of $F_{g}$). Moreover, since $\psi:\Hom\left(F_{g},Z\times G\right)\hookrightarrow\hom(\pi_{1},G)$
in \eqref{eq:inclusion} is clearly a $G$-equivariant inclusion of
affine algebraic varieties, in view of Equation \eqref{eq:schottkyrep}
and \eqref{eq:inclusion}, we have shown the following. 
\begin{prop}
\label{prop:SchottkyAndFreeGroup}There are the following morphisms
between algebraic $G$-varieties: 
\[
\mathcal{S}_{s}\cong\Hom(F_{g},G)\cong G{}^{g}\quad\subset\quad\mathcal{S}\cong\Hom(F_{g},Z\times G)\cong(Z\times G)^{g}\quad\subset\quad\hom(\pi_{1},G).
\]
In particular, $\mathcal{S}$ and $\mathcal{S}_{s}$ are smooth. In
turn, these induce morphisms of affine GIT quotients: 
\[
\mathbb{S}_{s}=\mathcal{S}_{s}\quot G\cong G^{g}\quot G\quad\subset\quad\mathbb{S}=\mathcal{S}\quot G\cong(Z\times G)^{g}\quot G\quad\subset\quad\mathbb{B}=\hom(\pi_{1},G)\quot G.
\]

\end{prop}
Note that, because the conjugation action is trivial on $Z$, we can
also write 
\begin{equation}
\mathbb{S}\cong(Z\times G)^{g}\quot G=Z^{g}\times(G{}^{g}\quot G)=Z^{g}\times\mathbb{S}_{s}.\label{eq:S=00003DZvezesS_s}
\end{equation}
The GIT quotient under $G$ of an irreducible variety is irreducible.
Thus, $\mathbb{S}_{s}\cong G^{g}\quot G$ is irreducible. However,
$\mathbb{S}$ can have several irreducible components, in bijection
with the components of $Z^{g}$. It is well known that the connected
component of the identity of $Z$ is an algebraic torus $Z^{\circ}$,
and the quotient $Z_{f}:=\left.Z\right/Z^{\circ}$ is finite. 
\begin{prop}
\label{Prop:many-irred-components} All irreducible components of
$\mathbb{S}$ are isomorphic to 
\[
\Hom\left(F_{g},Z^{\circ}\times G\right)\quot G\cong\left(Z^{\circ}\right)^{g}\times\left(G^{g}\quot G\right)\cong\left(Z^{\circ}\right)^{g}\times\mathbb{S}_{s},
\]
and the number of irreducible components of $\mathbb{S}$ is given
by $\left|Z_{f}\right|^{g}$. \end{prop}
\begin{proof}
As a variety, we can write $Z$ as a cartesian product of the above
subgroups, $Z=Z_{f}\times Z^{\circ}$. So, we get the following isomorphism
of varieties, from Equation \eqref{eq:S=00003DZvezesS_s} 
\[
\mathbb{S}\cong Z^{g}\times\mathbb{S}_{s}\cong\left(Z_{f}\right)^{g}\times\left(Z^{\circ}\right)^{g}\times\mathbb{S}_{s}\cong\left(Z_{f}\right)^{g}\times\Hom\left(F_{g},Z^{\circ}\times G\right)\quot G
\]
which immediately proves the proposition. \end{proof}
\begin{rem}
(1) Clearly, $\mathbb{S}=\mathbb{S}_{s}$, hence irreducible, when
the center of $G$ is trivial. \\
 (2) Replacing $F_{g}$ by other finitely generated groups can give
very different results on components. For example, when $G=PSL_{2}\mathbb{C}$
it is known that $\hom(\pi_{1},G)\quot G$ has several irreducible
components, and only two of them correspond to representations that
uniformize a Riemann surface (Kleinian representations). On the other
hand $\hom(\pi_{1},SL_{2}\mathbb{C})\quot SL_{2}\mathbb{C}$ is irreducible
(see \cite{GoldmanTopCompSR}). 
\end{rem}

\subsection{Good and unitary representations\label{sub:Good-and-unitary-rep}}

Although $\mathcal{S}$ and $\mathcal{S}_{s}$ are smooth, the algebraic
varieties $\mathbb{S}$ and $\mathbb{S}_{s}$ are singular in general.
The notion of a good representation allows us to consider smooth points
of the GIT quotient, as we will see. Let $\Gamma$ be a finitely generated
group, for example the fundamental group of a compact manifold. Given
a representation $\rho:\Gamma\to G$ we denote by 
\[
Z(\rho)=\{h\in G:\ \rho(\gamma)h=h\rho(\gamma)\ \forall\gamma\in\Gamma\}
\]
its stabilizer in $G$, and denote by $G\cdot\rho$ its $G$-orbit
in the algebraic variety $\hom(\Gamma,G)$. Recall the following standard
definitions. 
\begin{defn}
Let $\rho:\Gamma\to G$ be a representation. We say that $\rho$ is:

(a) \emph{polystable} if $G\cdot\rho$ is (Zariski)-closed,

(b) \emph{reducible} if $\rho(\Gamma)$ is contained in a proper parabolic
subgroup of $G$,

(c) \emph{irreducible} if it is not reducible,

(d) \emph{good} if $\rho$ is irreducible and $Z(\rho)=Z$. \end{defn}
\begin{rem}
Note that $\rho$ is polystable if and only if $Z(\rho)$ is a reductive
group itself, and it is irreducible if and only if $Z(\rho)$ is reductive
and a finite extension of $Z$ (see \cite{Sikora}). Moreover, $\rho$
is irreducible if and only if it is \emph{stable} in the appropriate
affine GIT sense (see \cite{FlorentinoCasimiro}).
\end{rem}
Now we apply these notions to the case of Schottky representations. 
\begin{defn}
\label{def: good Schottky rep}A representation $\rho\in\mathcal{S}\subset\Hom\left(\pi_{1},G\right)$
is said to be \emph{polystable}\textbf{ }(resp. \emph{irreducible},
\emph{good}) if $\rho$ is polystable (resp. irreducible, good) as
an element of $\Hom\left(\pi_{1},G\right)$. 
\end{defn}
Denote the set of all good\textbf{ }(resp. good Schottky) representations
by \textbf{$\Hom^{\mathsf{gd}}\left(\pi_{1},G\right)$} (resp. $\mathcal{S}^{\mathsf{gd}}$).
Since these notions are well defined under conjugation, we can define
the corresponding quotient spaces: 
\[
\mathbb{B}^{\mathsf{gd}}:=\Hom^{\mathsf{gd}}\left(\pi_{1},\,G\right)\quot G\quad\mbox{and}\quad\mathbb{S}^{\mathsf{gd}}:=\mathcal{S}^{\mathsf{gd}}\quot G,
\]
and, from Proposition \ref{prop:SchottkyAndFreeGroup}, we have the
inclusion $\mathbb{S}^{\mathsf{gd}}\subset\mathbb{B}^{\mathsf{gd}}$. 

The sets of good, polystable and irreducible representations are Zariski
open in $\mathcal{S}$ (see for example \cite{Sikora}). By \cite[Lemma 4.6]{Martin1}
there exists a good representation in $\Hom\left(\pi_{1},G\right)$,
that is, \textbf{$\Hom^{\mathsf{gd}}\left(\pi_{1},\,G\right)\neq\varnothing$},
if $X$ has genus $g\geq2$. Note that the case $g=1$ is slightly
different (see Section 9).

To show that \textbf{$\mathcal{S}^{\mathsf{gd}}$} is nonempty, we
start by relating the relevant properties of $\rho\in\mathcal{S}$
with the corresponding properties of $\rho_{2}:F_{g}\to G$. 
\begin{prop}
\label{prop:good Schottky as rho1 good }Let $\rho\in\mathcal{S}\subset\Hom\left(\pi_{1},G\right)$
be given by $\rho_{F}=\left(\rho_{1},\rho_{2}\right):F_{g}\to Z\times G$
as in (\ref{eq:schottkyrep}). Then:

(a) $Z(\rho)=Z(\rho_{2})\subset G$,

(b) $\rho$ is irreducible if and only if $\rho_{2}$ is irreducible,

(c) $\rho$ is a good Schottky representation if and only if $\rho_{2}$
is a good representation of $F_{g}$. \end{prop}
\begin{proof}
(a) Denote by $C(h)$ the centralizer of an element $h\in G$, $C(h):=\{g\in G:\ hg=gh\}$.
Since $\rho$ is completely defined by the image of the generators
of $\pi_{1}$, the stabilizer of $\rho$ is the intersection of the
centralizers of the images of the generators $\alpha_{i},\beta_{i}$
of $\pi_{1}$ and $\gamma_{i}$ of $F_{g}$: 
\[
Z\left(\rho\right)={\textstyle \bigcap_{i=1}^{g}C(\rho(\alpha_{i}))\bigcap_{i=1}^{g}C(\rho(\beta_{i}))=\bigcap_{i=1}^{g}C(\rho_{2}(\gamma_{i}))}=Z(\rho_{2}),
\]
because $\rho(\alpha_{i})=\rho_{1}\left(\gamma_{i}\right)\in Z$,
which implies $C(\rho(\alpha_{i}))=G$.

(b) Let us suppose that $\rho:\pi_{1}\to G$ is reducible. By definition,
$\rho\left(\pi_{1}\right)\subset P$ for some proper parabolic subgroup
$P\subset G$. This means that $\rho\left(\alpha_{i}\right),\rho\left(\beta_{i}\right)\in P,\,\forall i=1,\cdots,g$.
So, 
\[
\rho\left(\beta_{i}\right)=\rho_{2}\left(\gamma_{i}\right)\in P,\,\forall i\quad\Leftrightarrow\quad\rho_{2}\left(F_{g}\right)\subset P,
\]
proving that $\rho_{2}$ is reducible. The proof of the converse is
analogous, using again $\rho(\alpha_{i})=\rho_{1}(\gamma_{i})\in Z$,
and also the fact that any parabolic subgroup contains the center
of $G$.

(c) This follows immediately from (a) and (b). 
\end{proof}
Recall that, for a connected reductive algebraic group $G$ over $\mathbb{C}$,
there exists a maximal compact connected \emph{real} Lie group $K$
whose complexification coincides with $G$. Ramanathan showed that
the moduli space of semistable $G$-bundles over $X$, which admit
a flat connection, is homeomorphic to $\hom(\pi_{1},K)/K$ (\cite{Ramanathan1}).

We now show that good Schottky representations exist, and these can
be taken to be unitary, as well. 
\begin{lem}
\label{lem: density} Let $K$ be a maximal compact subgroup of $G$.
If $H$ is a subgroup of $K$ which is dense in the manifold topology
of $K$, then $Z_{G}(H)=Z_{G}(K)=Z$. \end{lem}
\begin{proof}
Being the intersection of centralizers of single elements, the centralizer
of any subgroup of $G$ is an algebraic subgroup of $G$, hence Zariski
closed. In particular, $Z_{G}(K)$ centralizes the Zariski closure
of $K$, which is well known to be $G$. So $Z_{G}(K)=Z_{G}(G)=Z$.
Moreover, since $H$ is dense in $K$, their centralizers are equal,
$Z_{G}(H)=Z_{G}(K)$. 
\end{proof}
Now recall that any connected compact Lie group can be generated by
two elements.\vspace{-0.15cm}
 
\begin{thm}
\cite{Auerbach}\label{thm: gen by two elmts dense in a compact subgroup}
Let $K$ be a connected compact Lie group. Then there are two elements
$c,\,d\in K$ such that the closure of the subgroup they generate,
$\overline{\left\langle c,d\right\rangle }$, equals $K$. Moreover,
the set of such pairs $\left\{ \left(c,\,d\right)\right\} $ is dense
in $K\times K$. \end{thm}
\begin{prop}
\label{prop: existence of good and unitary Schottky rep } Let $g\geq2$.
Then, there is always a good strict Schottky representation $\rho:\pi_{1}\to G$.
Moreover, such a representation can be defined to take values in $K$. \end{prop}
\begin{proof}
Let $c,\,d\in K$ be two elements of $K$, such that $\overline{\left\langle c,d\right\rangle }=K$,
as in Theorem \ref{thm: gen by two elmts dense in a compact subgroup}.
Then we explicitly define a unitary representation $\rho:\pi_{1}\to K$
by: 
\begin{align}
\rho\left(\alpha_{i}\right)=e,\quad\forall i=1,\cdots,g, & \quad\quad\mbox{and }\quad\left\{ \begin{array}{cc}
\rho\left(\beta_{1}\right)=c\\
\rho\left(\beta_{2}\right)=d\\
\rho\left(\beta_{i}\right)=e, & \forall i=3,\cdots,g,
\end{array}\right.\label{eq:const unitary Schottky rep}
\end{align}
Since the subgroup $H:=\left\langle c,d\right\rangle $ is dense in
$K$, the subgroup $\rho(\pi_{1})\subset K$ is also dense in $K$.
So $Z_{G}(\rho)=Z$, by Lemma \ref{lem: density}, which proves that
$\rho$ is a good strict Schottky representation. \end{proof}
\begin{thm}
\label{thm:good rep for g2} Let $g\geq2$. The subsets of good representations\textbf{
$\Hom^{\mathsf{gd}}\left(\pi_{1},\,G\right)$} and $\mathcal{S}^{\mathsf{gd}}$
are Zariski open in \textbf{$\Hom\left(\pi_{1},\,G\right)$} and $\mathcal{S}$,
respectively. A good representation defines a smooth point in the
corresponding geometric quotient. Thus, the geometric quotients $\mathbb{B}^{\mathsf{gd}}$
and $\mathbb{S}^{\mathsf{gd}}$ are complex manifolds, and $\mathbb{S}^{\mathsf{gd}}$
is a complex submanifold of $\mathbb{B}^{\mathsf{gd}}$. \end{thm}
\begin{proof}
In Proposition \ref{prop: existence of good and unitary Schottky rep }
we constructed a good Schottky representation, for $g\geq2$. By \cite[Proposition 33]{Sikora},
the subspaces of good representations in $\hom(\pi_{1},G)$ and $\mathcal{S}$
are Zariski open. Thus, \textbf{$\Hom^{\mathsf{gd}}\left(\pi_{1},\,G\right)$}
and $\mathcal{S}^{\mathsf{gd}}$ are open. Since we are considering
either surface groups or free groups, \cite[Corollary 50]{Sikora}
shows that if $\rho\in\Hom^{\mathsf{gd}}\left(\pi_{1},\,G\right)$,
respectively $\rho\in$$\mathcal{S}^{\mathsf{gd}}$, then its class
$[\rho]$ is a smooth point of $\mathbb{B}$, respectively $\mathbb{S}$. 
\end{proof}

\section{Higgs bundles and Schottky representations\label{sec:Higgs-bundles,-Schottky,Branes}}

In this section, we relate Schottky representations to certain Lagrangian
subspaces of the moduli space of Higgs $G$-bundles. It is a fundamental
result in the theory of Higgs bundles, the so-called non-abelian Hodge
theorem, that by considering the Hitchin equations for $G$-Higgs
fields, one obtains a homeomorphism between the Betti space $\mathbb{B}=\hom(\pi_{1},G)\quot G$
and the moduli space of semistable $G$-Higgs bundles over $X$, denoted
by $\mathcal{H}$.

It is a recent observation in \cite{BaragliaSchaposnick} that, when
considering $G$-Higgs bundles over Riemann surfaces with a real structure,
one is naturally lead to representations into $G$ of the fundamental
group of a 3-manifold with boundary $X$. These are naturally related
to Schottky representations, as we present below. Our approach via
Schottky representations has one advantage: by showing the vanishing
of the complex symplectic form on the strict Schottky locus (see Proposition
\ref{prop:Schottky-lagrangian}), we get a simple argument for the
fact that (at least a natural component of) the Baraglia-Schaposnik
brane is indeed non-empty and Lagrangian with respect to the natural
complex structure of $\mathbb{B}$ (coming from the complex structure
of $G$).

\subsection{Schottky representations and flat connections on a three manifold.}

Suppose that our Riemann surface $X$, of genus $g$, is the boundary
$\partial M$, of a compact 3-manifold $M$. Choose a basepoint in
this boundary, $x_{0}\in X\subset M$. From the inclusion of pointed
spaces $(X,x_{0})\hookrightarrow(M,x_{0})$ one gets an induced homomorphism:
\begin{equation}
\varphi\,:\,\pi_{1}=\pi_{1}(X,x_{0})\to\pi_{1}(M,x_{0}),\label{eq:phi}
\end{equation}
between their fundamental groups.

One particularly interesting case is when $X$ bounds a 3-dimensional
handlebody $M$, so that $\pi_{1}(M,x_{0})$ is free of rank $g$.
In this case, by carefully choosing the generators of each fundamental
group, we can arrange so that $\varphi$ coincides with the map defining
Schottky representations (see Lemma \ref{lem:alternative-def}). 
\begin{prop}
\label{prop:Schottky-flat}Let $M$ be a compact 3-dimensional handlebody
of genus $g$ whose boundary is a compact surface $X$. Then, the
moduli space $\mathbb{S}_{s}$ of strict Schottky representations
with respect to $\varphi$ coincides with the moduli space $\mathbb{F}_{M}(G)$
of flat $G$-connections over $M$. \end{prop}
\begin{proof}
By hypothesis $\pi_{1}(M,x_{0})$ is a free group of rank $g$, and
$\pi_{1}$ has a ``symplectic presentation'' in terms of generators
$\alpha_{i}$ and $\beta_{i}$, $i=1,\cdots,g$, as in Equation \eqref{eq:fund-group},
so that 
\[
\varphi(\alpha_{i})=1,\quad\quad\varphi(\beta_{i})=\gamma_{i},\quad\quad i=1,\cdots,g,
\]
where $\gamma_{1},\cdots,\gamma_{g}$ form a free basis of $\pi_{1}(M,x_{0})$.
Thus, a strict Schottky representation $\rho:\pi_{1}\to G$ with respect
to $\varphi$ factors through a representation of $\pi_{1}(M,x_{0})\cong F_{g}$
via $\varphi$. By standard differential geometry arguments, this
is precisely the same as saying that the corresponding flat connection
$\nabla_{\rho}$ on $X$ extends, as a flat connection, to the 3-manifold
$M$. Conversely, a flat $G$-connection on $M$ induces a representation
$\rho:\pi_{1}\to G$ satisfying $\rho(\ker\varphi)=\{e\}$, and thus
it is a strict Schottky representation of $\pi_{1}$ (with respect
to $\varphi$), by Lemma \ref{lem:alternative-def}. This correspondence
is well defined up to conjugation by $G$, and so, we have a natural
identification: 
\[
\mathbb{S}_{s}=\Hom(F_{g},G)\quot G\cong\mathbb{F}_{M}(G),
\]
as wanted. 
\end{proof}

\subsection{Schottky representations and $(A,B,A)$-branes}

Suppose now that we have an anti-holomorphic involution $f:X\to X$,
defining a real structure on $X$. This induces, as in \cite[§3]{BaragliaSchaposnick},
an anti-holomorphic involution 
\begin{equation}
f^{*}:\mathcal{H}\to\mathcal{H},\label{eq:anti-hol-inv}
\end{equation}
where $\mathcal{H}$ is the moduli space of $G$-Higgs bundles over
$X$. Following \cite[§3]{BaragliaSchaposnick}, denote the set of
fixed points of $f^{*}$ in $\mathbf{\mathcal{H}}$ by $\mathcal{L}_{G}$,
and call it the \emph{Baraglia-Schaposnik brane} inside $\mathcal{H}$.

Consider the 3-manifold with boundary $\hat{X}:=X\times[-1,1]$. The
anti-holomorphic involution $f:X\to X$ defines now an orientation
preserving involution $\sigma:\hat{X}\to\hat{X}$ given by 
\[
\sigma(x,t)=(f(x),-t).
\]
Note that the boundary of $\hat{X}$ consists of two copies of $X$,
but the boundary of the compact 3-manifold $M:=\hat{X}/\sigma$ is
homeomorphic to $X$.
\begin{prop}
\label{prop:Schottky-brane}Let $f:X\to X$ be an anti-holomorphic
involution such that $M$ is a handlebody of genus $g$, and let $x_{0}\in X\subset M$
be fixed by $f$. Then, the moduli space $\mathbb{S}_{s}$ of strict
Schottky representations with respect to the map $\varphi$ in \eqref{eq:phi}
is included in the Baraglia-Schaposnik brane $\mathcal{L}_{G}$. \end{prop}
\begin{proof}
In \cite[Prop. 43]{BaragliaSchaposnick}, Baraglia and Schaposnik
show that any flat $G$-connection on $M$ defines, under the non-abelian
Hodge theorem sending $\mathbb{B}$ to $\mathcal{H}$, a $G$-Higgs
bundle which is fixed by the involution $f^{*}$. Thus, they have
produced a map, which they prove to be an inclusion: 
\[
\mathbb{F}_{M}(G)\to\mathcal{L}_{G}\subset\mathcal{H}.
\]
Since, by Proposition \ref{prop:Schottky-flat}, $\mathbb{S}_{s}$
can be identified with $\mathbb{F}_{M}(G)$ the proposition follows. 
\end{proof}
\begin{figure}
\begin{minipage}[b]{1\columnwidth}%
\begin{center}
\includegraphics[scale=0.22]{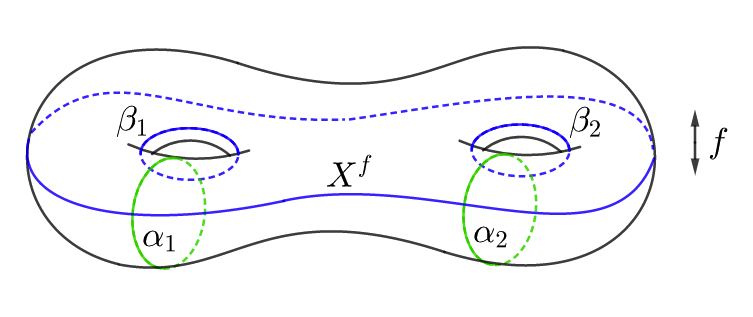}$\quad\quad\qquad$ \includegraphics[scale=0.12]{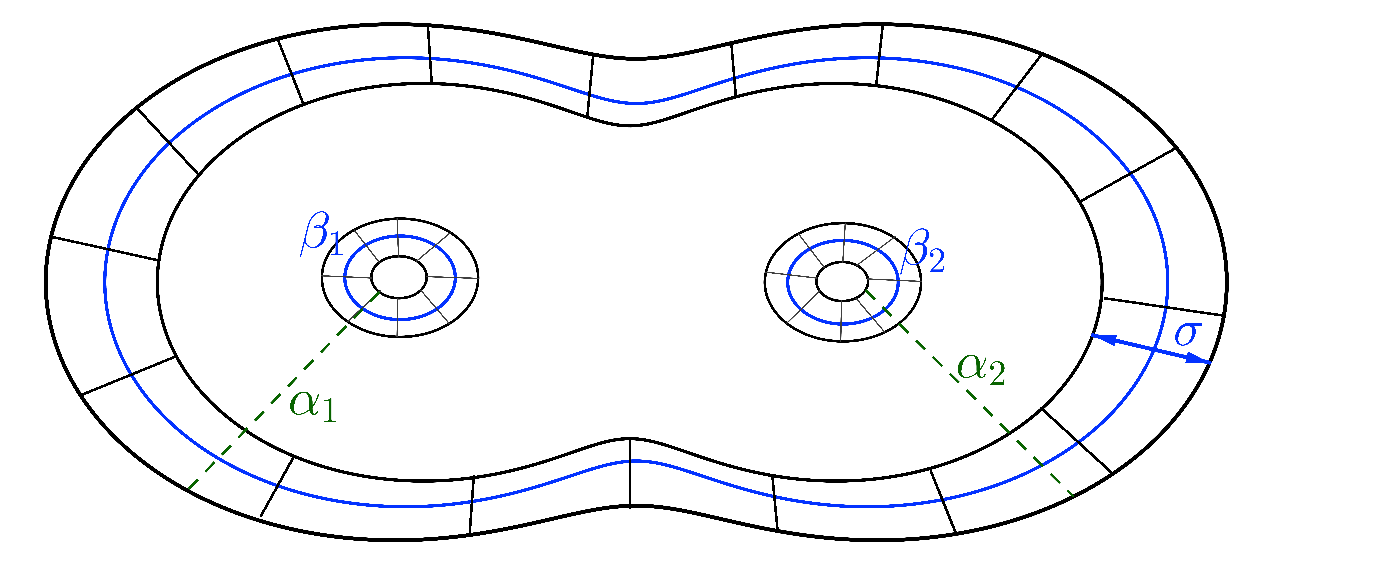}\\
 {\small{}(A) Involution $f$ on $X$, and its fixed curves $X^{f}$$\quad\quad\quad$
(B) Involution $\sigma$ restricted to $X^{f}\times I$$\quad$\\[30mm]} 
\par\end{center}%
\end{minipage}

\caption{\label{fig:Involucoes}Involutions}
\end{figure}
\vspace{-1mm}

\begin{rem}
\label{rem:g+1-loops}The assumption of the previous proposition is
verified when the anti-holomor\-phic involution $f$ has as fixed
point locus, $X^{f},$ the union of $g+1$ disjoint loops and disconnected
orientation double cover (see \cite[Proposition 3]{BaragliaSchaposnick}
and Figure \ref{fig:Involucoes}). In this case, \cite[Proposition 10]{BaragliaSchaposnick}
says that the set of smooth points of $\mathcal{L}_{G}$ is a non-empty
Lagrangian submanifold of $\mathbf{\mathcal{H}}$. In a future work,
we plan to further address this construction. 
\end{rem}

\section{Schottky $G$-bundles \label{sec:Schottky-G-bundles}}

Let again $X$ be a compact Riemann surface, with fundamental group
$\pi_{1}$ and $\rho:\pi_{1}\to G$ be a representation into a reductive
group. The associated bundle construction, from a universal cover
$p:Y\to X$, defines a $G$-bundle over $X$ associated to $\rho$.
We write this $G$-bundle as $E_{\rho}:=(Y\times G)/_{\!\!\rho}\,\pi_{1}$,
with equivalence classes given by\footnote{We are using a left action both on $Y$ and on $G$; this was chosen
(other options would be equivalent) for a standard use of Fox calculus
in section \ref{sec:Surjectivity-of-the}.}
\begin{equation}
(y,~g)\sim\left(\gamma\cdot y,~\rho(\gamma)\cdot g\right),~~~~~~\forall\gamma\in\pi_{1},\,(y,g)\in Y\times G.\label{eq:E_rho}
\end{equation}

\subsection{Schottky principal bundles and the uniformization map }

Thus, the space of representations parametrizes holomorphic $G$-bundles,
and we can view this construction as providing a natural map, that
we call the \emph{uniformization map}: 
\begin{equation}
\begin{array}{cccc}
\mathbf{E}: & \mathbb{B} & \to & M_{G}\\
 & [\rho] & \mapsto & [E_{\rho}]
\end{array}\label{eq:uniformization-map}
\end{equation}
Here, $M_{G}$ represents the set of isomorphism classes of $G$-bundles
that admit a holomorphic flat connection. To simplify terminology,
we say that a bundle \emph{is flat} if it admits a holomorphic flat
connection. Note that $\mathbf{E}$ is well defined on conjugacy classes,
since if $\rho$ and $\sigma$ are conjugate representations, then
$E_{\rho}\cong E_{\sigma}$. Moreover, by considering the holonomy
representation of a given flat $G$-bundle, the map $\mathbf{E}$
is easily seen to be surjective.%

\begin{defn}
A $G$-bundle $E$ over the Riemann surface $X$ is called: \end{defn}
\begin{enumerate}
\item a\textbf{ }\emph{Schottky $G$-bundle} if $E$ is isomorphic to $E_{\rho}$
for some Schottky representation $\rho:\pi_{1}\rightarrow G$, that
is, $\rho\left(\alpha_{i}\right)\in Z$ for all $i=1,\cdots,g$. 
\item a\textbf{ }\emph{strict Schottky $G$-bundle} if $E$ is isomorphic
to $E_{\rho}$ for some strict Schottky representation $\rho:\pi_{1}\rightarrow G$,
that is, $\rho\left(\alpha_{i}\right)=e$ for all $i=1,\cdots,g$. \end{enumerate}
\begin{rem}
(1) \label{rem: SchVecBd}\textbf{ }\emph{Schottky vector bundles}
were defined by \cite{Florentino} as vector bundles isomorphic to
$V_{\rho}:=(Y\times\mathbb{C}^{n})/_{\!\!\rho}\,\pi_{1}$ for a representation
$\rho:\pi_{1}\to GL_{n}\mathbb{C}$ with $\rho\left(\alpha_{i}\right)=e$
for all $i=1,\cdots,g$. Then, the associated \emph{frame bundle}
is, by definition the $GL_{n}\mathbb{C}$-bundle defined by the same
representation: $E_{\rho}=(Y\times GL_{n}\mathbb{C})/_{\!\!\rho}\,\pi_{1}$.
So, if $V$ is a Schottky vector bundle then the associated frame
bundle is a \emph{strict} Schottky $GL_{n}\mathbb{C}$-bundle. In
other words, according to our definition, Schottky vector bundles
are the same as strict Schottky (principal) $GL_{n}\mathbb{C}$-bundles.
See, however, Proposition \ref{connected-center} and Example \ref{exa:SL_n}.\\
 (2) In terms of the uniformization map in Equation \eqref{eq:uniformization-map}
we simply say that $E$ is Schottky (resp. strict Schottky) if and
only if $\mathbf{E}^{-1}([E])\subset\mathbb{S}$ (resp. $\mathbf{E}^{-1}([E])\subset\mathbb{S}_{s}$). 
\end{rem}

\subsection{Associated Schottky bundles}

In the following, we describe how the Schottky property is transferred
to associated bundles. Throughout this section, $G$ and $H$ denote
connected reductive algebraic groups, $Z_{G}$ and $Z_{H}$ the corresponding
centers.

Suppose we have a $G$-bundle $E$ over $X$. Then, the $H$-bundle
over $X$, obtained from the trivial bundle $E\times H\to E$ by letting
$G$ act on $H$ through a homomorphism $\phi:G\to H$ is denoted
by $E(H):=(E\times H)/_{\!\!\phi}\,G$, and we say that $E(H)$ is
obtained from $E$ by \emph{extension of structure group}. This is,
conceptually, the same as the construction of the bundle $E_{\rho}$
starting from universal cover of $X$, the $\pi_{1}$ bundle $Y\to X$,
and the homomorphism $\rho:\pi_{1}\to G$, as in \eqref{eq:E_rho}. 
\begin{prop}
\label{Schottky and group extension} Let $\phi:G\to H$ be a group
homomorphism and $E$ be a Schottky $G$-bundle. Then:
\begin{enumerate}
\item If $E$ is a \emph{strict} Schottky $G$-bundle, then $E(H)$ is a
strict Schottky $H$-bundle. 
\item If $\phi(Z_{G})\subset Z_{H}$, then $E(H)$ is a Schottky $H$-bundle. 
\end{enumerate}
\end{prop}
\begin{proof}
First note that if $E=E_{\rho}$, for some $\rho:\pi_{1}\to G$, then
$E(H)=E_{\phi\circ\rho}$. Then, assuming $\rho$ is a strict Schottky
representation, $\rho(\ker\varphi)$ is the identity of $G$ (as in
Lemma \ref{lem:alternative-def}). This implies that $(\phi\circ\rho)(\ker\varphi)=\phi(e)=e_{H}$,
the identity of $H$, so $E_{\phi\circ\rho}$ is a strict Schottky
bundle, as wanted. The second case is similar, using the hypothesis
$\phi(Z_{G})\subset Z_{H}$. 
\end{proof}
A $G\times H$-bundle $E$ can be seen as an ordered pair $(E_{G},E_{H})$,
with $E_{G}$ and $E_{H}$ a $G$-bundle and a $H$ bundle respectively.
Indeed, from $E$ we can define $E_{G}:=E(G)$ and $E_{H}:=E(H)$,
where there are considered the projections $\pi_{G}:G\times H\to G$
and $\pi_{H}:G\times H\to H$, respectively. So, the following proposition
is an easy consequence of the previous one. 
\begin{prop}
\label{prop:product-bundles}A $(G\times H)$-bundle $E$ is (strict)
Schottky if and only if the $E_{G}$ and $E_{H}$ are (strict) Schottky
principal bundles. \end{prop}
\begin{proof}
Assume that $E=E_{\rho}$ for a certain Schottky representation $\rho=(\rho_{G},\rho_{H}):\pi_{1}\to G\times H$.
Then, both $\rho_{G}$ and $\rho_{H}$ are Schottky representations
because $(\rho_{G}(\alpha_{i}),\rho_{H}(\alpha_{i}))=\rho(\alpha_{i})\in Z_{G\times H}=Z_{G}\times Z_{H}$
for $i=1,\cdots,g$. By Proposition \ref{Schottky and group extension},
$E_{\pi_{G}\circ\rho}$ is Schottky. On the other hand, it is easy
to see that $E_{G}=E_{\rho_{G}}=E_{\pi_{G}\circ\rho}$, so $E_{G}$
is Schottky. The same argument applies to $E_{H}$. The converse statement
and the strict case are treated in a similar fashion. 
\end{proof}
Let now $\mathfrak{g}=\mathsf{Lie}(G)$ be the Lie algebra of $G$.
Given a $G$-bundle $E$, the $GL(\mathfrak{g})$-bundle associated
to the adjoint representation $\Ad:G\rightarrow GL(\mathfrak{g})$
corresponds, via the frame bundle construction, to the vector bundle
(with $\mathfrak{g}$ as fiber): 
\begin{equation}
\Ad(E):=E\times_{\Ad}\mathfrak{g}.\label{eq: adjoint bdl Ad(E_G)}
\end{equation}
called the\emph{ adjoint bundle}. %

\begin{prop}
\label{EToAdESchottky} If $E$ is a Schottky $G$-bundle then the
adjoint bundle $\Ad(E)$ is a Schottky vector bundle. \end{prop}
\begin{proof}
Since $E$ is a Schottky $G$-bundle, there is $\rho\in\hom(\pi_{1},G)$
with $\rho\left(\alpha_{i}\right)\in Z$ for all $i=1,\cdots,g$,
such that $E\cong E_{\rho}=(Y\times G)/_{\!\!\rho}\,\pi_{1}$. By
construction, the vector bundle associated to $E$ by the adjoint
representation can be seen as 
\[
\Ad(E)=E\times_{\Ad}\mathfrak{g}\cong(Y\times\mathfrak{g})/_{\!\!\Ad_{\rho}}\,\pi_{1}
\]
where $\Ad_{\rho}:\pi_{1}\rightarrow G\to GL(\mathfrak{g})$ is the
composition of the representations $\Ad$ and $\rho$. Because $\rho(\alpha_{i})\in Z$
and since $\ker\left(\Ad\right)=Z$, we see that $\Ad_{\rho}(\alpha_{i})$
is the identity map, for all $i=1,\cdots,\,g$. Thus, we obtain a
strict Schottky representation $\Ad_{\rho}:\pi_{1}\to GL(\mathfrak{g})$.
So, $\Ad(E)\cong Y\times_{\Ad\rho}\mathfrak{g}$ is a strict Schottky
$GL(\mathfrak{g})$-bundle. 
\end{proof}
The following simple example shows that the converse of Proposition
\ref{EToAdESchottky} is not valid. 
\begin{example*}
Consider the $\mathbb{C}{}^{*}$-bundle $E\to X$ defined as the frame
bundle of a line bundle $L$ with non-zero first Chern class. Then,
$\Ad\left(E\right)$ is the trivial line bundle, as conjugation is
trivial in this case, so that $\Ad\left(E\right)$ is trivially a
Schottky vector bundle. But $E$ is not Schottky, as it does not admit
a flat holomorphic connection (Weil's theorem \cite{Weil}). 
\end{example*}
However, by only requiring that $E$ admits a flat holomorphic connection,
we obtain a necessary and sufficient condition. 
\begin{prop}
\label{AdEToESchottky} Suppose that the $G$-bundle $E$ admits a
flat holomorphic connection. Then, $E$ is a Schottky $G$-bundle
if and only if $\Ad(E)$ is a Schottky vector bundle. \end{prop}
\begin{proof}
If $E$ is Schottky, Proposition \ref{EToAdESchottky} implies that
$\Ad(E)$ is Schottky. Conversely, suppose that $E$ admits a flat
$G$-connection. Then it is of the form $E\cong E_{\rho},$ for some
$\rho:\pi_{1}\to G$. Note that $\Ad(E_{\rho})\cong E_{\Ad_{\rho}}.$
Since by hypothesis $\Ad(E)$ is a Schottky vector bundle, this means
that $\Ad_{\rho}(\alpha_{i})$ is the identity morphism, $\forall i=1,\cdots,g$.
As $\ker(\Ad)=Z$ (because $G$ is reductive), we may conclude that
$\rho(\alpha_{i})\in Z$ for all $i=1,\cdots,g$, that is, $E\cong E_{\rho}$
where $\rho$ is a Schottky representation. 
\end{proof}
Moreover, when $G$ is a connected semisimple algebraic group, we
can drop the flatness condition above. 
\begin{thm}
\label{GSemisimpleSchottky-1} Let $G$ be a connected semisimple
algebraic group. Then $E$ is a Schottky $G$-bundle if and only if
the adjoint bundle $\Ad(E)$ is a Schottky vector bundle. \end{thm}
\begin{proof}
By Proposition \ref{EToAdESchottky}, if $E$ is Schottky, $\Ad(E)$
is Schottky too. Conversely, assume that $\Ad(E)$ is a Schottky vector
bundle. Then, $\Ad(E)$ admits a flat connection and \cite[Proposition 2.2]{AzadBiswas}
proved that, because $G$ is semisimple, $E$ admits a flat connection
too. So, the conditions of Proposition \ref{AdEToESchottky} are fulfilled,
and $E$ is a Schottky $G$-bundle. 
\end{proof}

\section{Topological type \label{sec:Topological-type}}

The moduli space of $G$-bundles over a compact Riemann surface is
a disjoint union of connected components indexed by $\pi_{1}\left(G\right)$,
the fundamental group of $G$ (see \cite{Ramanathan1}, \cite{GarciaOliveira}).
In this section, we show that all Schottky $G$-bundles over a compact
Riemann surface $X$ have trivial topological type, corresponding
to the identity element $0\in\pi_{1}\left(G\right)$. Therefore, any
Schottky $G$-bundle $E$ is globally trivial in the smooth category,
although it is generally non-trivial as a flat, or as an algebraic
principal bundle.

\subsection{Topological types of $G$-bundles}

In this subsection, $G$ is just a \emph{connected topological group}
which admits a universal cover (this is the case provided $G$ is
locally path connected and semilocally simply connected). To characterize
$G$-bundles topologically, consider the short exact sequence of group
homomorphisms 
\begin{equation}
1\rightarrow\ker p\rightarrow\widetilde{G}\overset{\,\,p\,\,}{\rightarrow}G\rightarrow1,\label{eq:short-exact-sequence}
\end{equation}
where $p:\widetilde{G}\to G$ is a universal cover. It is known that
$\ker p\cong\pi_{1}\left(G\right)$ is a discrete subgroup of the
center of $\widetilde{G}$, so that \eqref{eq:short-exact-sequence}
defines $\widetilde{G}$ as a central extension of $G$ (cf. also
Lemma \ref{lem:centro} below). The exact sequence \eqref{eq:short-exact-sequence}
induces a short exact sequence of sheaves 
\[
1\rightarrow\pi_{1}(G)\rightarrow\underline{\widetilde{G}}\overset{\,\,p\,\,}{\rightarrow}\underline{G}\rightarrow1,
\]
where the underline denotes the sheaf of continuous functions defined
on open subsets of the base $X$ into the corresponding group. In
turn, we get an exact sequence in (non-abelian) sheaf cohomology,
with an associated coboundary map: 
\[
H^{1}\left(X,\:\underline{G}\right)\stackrel{\delta}{\longrightarrow}H^{2}\left(X,\,\pi_{1}(G)\right)\cong\pi_{1}(G),
\]
whose right isomorphism comes from using the orientation on $X$ (see,
for example \cite{GoldmanTopCompSR}). The map $\delta$ serves to
define the \emph{topological type} of a $G$-bundle. Namely, interpreting
an isomorphism class of a $G$-bundle $E$ as an element of $H^{1}\left(X,\:\underline{G}\right)$
we define its topological type as (see also \cite[Remark 5.2]{Ramanathan1})
\[
\delta(E):=\delta([E])\in\pi_{1}(G).
\]
The topological type is functorial in the sense that, if a $H$-bundle
$E_{H}$ is obtained from a $G$-bundle $E_{G}$ by extension of the
structure group $\phi:G\to H$, then: 
\begin{equation}
\delta(E_{H})=\phi_{*}(\delta(E_{G})),\label{eq:extension}
\end{equation}
using the induced morphism $\phi_{*}:\pi_{1}(G)\to\pi_{1}(H)$ (see
\cite[Remark 5.1]{Ramanathan1}). The following simple lemma should
be well known, but we include a proof for convenience of the reader.
\begin{lem}
\label{lem:centro} Let $G$ be a connected, locally path connected
and semilocally simply connected topological group, and $p:\widetilde{G}\to G$
be a universal cover of $G$. Then $Z_{\widetilde{G}}=p^{-1}\left(Z_{G}\right)$
and $p\left(Z_{\widetilde{G}}\right)=Z_{G}$. \end{lem}
\begin{proof}
Let $\tilde{z}\in Z_{\widetilde{G}}$. Since $p$ is surjective, for
all $h\in G$, there is $\tilde{h}\in p^{-1}(h)\subset\widetilde{G}$,
and we obtain 
\[
hp(\tilde{z})=p(\tilde{h})p(\tilde{z})=p(\tilde{h}\tilde{z})=p(\tilde{z})p(\tilde{h})=p(\tilde{z})h,
\]
showing that $p\left(\tilde{z}\right)\in Z_{G}$. We conclude that
$Z_{\widetilde{G}}\subset p^{-1}(Z_{G})$.

Conversely, let $z\in Z_{G}$ and fix $\tilde{z}\in p^{-1}(z)\subset\widetilde{G}$.
We want to show that $\tilde{z}\in Z_{\widetilde{G}}$. Since $\widetilde{G}$
is path connected, given $\tilde{h}\in\widetilde{G}$ there is a continuous
path $\lambda:[0,1]\to\widetilde{G}$ with $\lambda(0)=\tilde{e}$
and $\lambda(1)=\tilde{h}$, where $\tilde{e}$ is the identity element
of $\widetilde{G}$. Since $p(\tilde{z}\lambda(t)\tilde{z}^{-1}\lambda(t)^{-1})=zp(\lambda(t))z^{-1}p(\lambda(t))^{-1}=e$,
the following map is well defined and continuous: 
\[
\begin{array}{ccc}
\Psi:[0,1] & \to & \ker p\\
t & \mapsto & \tilde{z}\lambda(t)\tilde{z}^{-1}\lambda(t)^{-1}.
\end{array}
\]
Noting that $\ker p\cong\pi_{1}(G)\subset Z_{\widetilde{G}}$ is a
discrete subgroup of $\widetilde{G}$, the image of $\Psi$ is constant
and so $\Psi([0,1])=\{\tilde{e}\}$. Thus: 
\[
\tilde{e}=\Psi(0)=\Psi(1)=\tilde{z}\tilde{h}\tilde{z}^{-1}\tilde{h}^{-1},
\]
showing that $\tilde{z}\in Z_{\widetilde{G}}$. Finally $p\left(Z_{\widetilde{G}}\right)=Z_{G}$
is a simple consequence of $Z_{\widetilde{G}}=p^{-1}\left(Z_{G}\right)$. 
\end{proof}

\subsection{Topological triviality of Schottky $G$-bundles}

Now, we return to the case where $G$ is a \emph{connected complex
reductive group}, and suppose that $E$ is a flat $G$-bundle, a bundle
isomorphic to $E_{\rho}$ for some $\rho:\pi_{1}\to G$. Then, the
value $\delta(E)$ lies, in fact, in the subgroup $\pi_{1}(DG)\subset\pi_{1}(G)$
coming from the natural inclusion $DG\hookrightarrow G$, where $DG$
is the derived group of $G$. Moreover, in \cite{Ramanathan1}, Ramanathan
defined a natural map from connected components of $\hom(\pi_{1},G)$
to $\pi_{1}(DG)$. More precisely, the following statement was recently
shown in \cite[Appendix]{LawtonRamrasHoLiu} (following \cite{Li}
and \cite{Ramanathan1}).
\begin{thm}
For any complex reductive group $G$, there is a natural bijection
\[
\pi_{0}(\hom(\pi_{1},G))\cong\pi_{1}(DG).
\]

\end{thm}
In particular, for groups whose derived group is not simply connected,
there exist \emph{flat bundles} $E$ which are \emph{not topologically
trivial}. By contrast, all Schottky bundles are topologically trivial,
as we now show.
\begin{thm}
\label{thm:Schottky bdl has trivial type} Let $G$ be a connected
complex reductive group, and let $E$ be a Schottky $G$-bundle. Then
$E$ has trivial topological type. \end{thm}
\begin{proof}
If $E\cong E_{\rho}$ is Schottky then it is defined by a representation
$\rho:F_{g}\rightarrow Z_{G}\times G$. Since $F_{g}$ is free, we
can lift $\rho$ to a representation $\widetilde{\rho}:F_{g}\to\widetilde{G}\times\widetilde{G}$,
verifying $\rho=(p\times p)\circ\widetilde{\rho}$, where as above,
$p:\widetilde{G}\to G$ is the universal cover ($\widetilde{G}$ is
a complex Lie group, not necessarily algebraic). By Lemma \ref{lem:centro},
we see that $\widetilde{\rho}$ has image in $Z_{\widetilde{G}}\times\widetilde{G}\subset\widetilde{G}\times\widetilde{G}$,
so it defines a Schottky representation inside $\hom(\pi_{1},\widetilde{G})$.
The corresponding $\widetilde{G}$-bundle $E_{\widetilde{\rho}}$
is topological trivial, since $\pi_{1}(\widetilde{G})$ is trivial.
Finally, as $E_{\rho}=E_{(p\times p)\circ\widetilde{\rho}}$ is obtained
by extension of structure group, it is also topological trivial, by
Equation \eqref{eq:extension} with $\phi=p\times p$.
\end{proof}
By \cite[Theorem 5.9]{Ramanathan1}, the components of the moduli
space of \emph{semistable} $G$-bundles $\mathcal{M}_{G}$ over a
Riemann surface $X$ are normal projective varieties indexed by the
topological types of $G$-bundles. Thus, we can write the moduli space
$\mathcal{M}_{G}$ as a disjoint union 
\[
\mathcal{M}_{G}=\bigsqcup_{\delta\in\pi_{1}(G)}\mathcal{M}_{G}^{\delta}.
\]
where $\mathcal{M}_{G}^{\delta}$ denotes the moduli space of semistable
$G$-bundles with topological type $\delta\in\pi_{1}(G)$. In \cite{GarciaOliveira},
the theory of $G$-Higgs bundles is used to prove the non-emptiness
of the moduli spaces $\mathcal{M}_{G}^{\delta}$, for each topological
type $\delta\in\pi_{1}(G)$ (see also \cite[Proposition 7.7]{Ramanathan2}).
In particular, $\mathcal{M}_{G}^{0}\subset\mathcal{M}_{G}$ is connected. 
\begin{cor}
\label{cor: mod sp ss Schottky is contained conn comp} The isomorphism
class of a semistable Schottky $G$-bundle $E$ lies in the connected
component $\mathcal{M}_{G}^{0}$.
\end{cor}

\section{The Uniformization Map\label{sec:Uniformization-map}}

The association of a $G$-bundle to a representation of $\pi_{1}$
was called the uniformization map in section \ref{sec:Schottky-G-bundles}.
In this section, we introduce the notion of analytic equivalence (see
\cite{Florentino}), consider the tangent space of Schottky space
at good representations, and define the period map.

\subsection{Analytic equivalence}

Recall that the uniformization map \eqref{eq:uniformization-map}
\begin{equation}
\begin{array}{cccc}
\mathbf{E}: & \mathbb{B}:=\Hom\left(\pi_{1},G\right)\quot G & \to & M_{G}\\
 & [\rho] & \mapsto & [E_{\rho}]
\end{array}\label{eq: map E-1}
\end{equation}
is surjective but, in general, not injective. This leads us to consider
what we call \emph{analytic equivalence}. 
\begin{defn}
Two representations $\rho,\,\sigma\in\Hom\left(\pi_{1},G\right)$
are called \emph{analytically equivalent} if their associated $G$-bundles
are isomorphic, so that $E_{\rho}\cong E_{\sigma}$, or equivalently
$\mathbf{E}[\rho]=\mathbf{E}[\sigma]$.
\end{defn}
The next result provides two useful criteria for analytic equivalence,
generalizing Lemma 2 of (\cite{Florentino}) (see also \cite{Gunning}),
one of them in terms of holomorphic sections of $\Omega_{X}^{1}$,
the canonical line bundle of $X$. Let $p:Y\to X$ be a universal
covering map of $X$. 
\begin{thm}
\label{analytic equivalence G bundles} Let $\rho,\,\sigma\in\Hom(\pi_{1},\,G)$
and $y_{0}\in Y$. Then the following conditions are equivalent:
\begin{enumerate}
\item $E_{\sigma}\cong E_{\rho}$, that is $\sigma$ and $\rho$ are analytically
equivalent; 
\item There exists a holomorphic function $h:Y\rightarrow G$ such that
\[
h(\gamma\cdot y)=\rho(\gamma)h(y)\sigma(\gamma)^{-1},\quad\forall\gamma\in\pi_{1},y\in Y;
\]

\item There exists $\omega\in H^{0}(X,\mathrm{Ad}(E_{\sigma})\otimes\Omega_{X}^{1})$
such that 
\[
\sigma(\gamma)=h_{\omega}(\gamma\cdot y)\rho(\gamma)h_{\omega}(y)^{-1},\quad\forall\gamma\in\pi_{1},y\in Y
\]
where $h_{\omega}$ is the unique solution of the differential equation
$h^{-1}dh=\omega$ with the initial condition $h(y_{0})=e\in G$. 
\end{enumerate}
\end{thm}
\begin{proof}
$(1)\Leftrightarrow(2)$ Since the pullback $p^{*}\left(E_{\sigma}\right)\to Y$
using $p:Y\to X$ is a holomorphically trivial $G$-bundle on $Y$,
its sections $s_{\sigma}$ can be viewed as holomorphic maps $s_{\sigma}:Y\to G$
satisfying $s_{\sigma}(\gamma\cdot y)=\sigma(\gamma)s_{\sigma}(y)$
for all $\gamma\in\pi_{1},\,y\in Y$ (and similarly for $E_{\rho}$).
Analogously, an isomorphism $\psi:E_{\sigma}\to E_{\rho}$ is given
by an isomorphism between the pullback bundles $\tilde{\psi}:p^{*}(E_{\sigma})\to p^{*}(E_{\rho})$
satisfying $\tilde{\psi}(y,g)=(y,\,h(y)\,g),\ \forall(y,g)\in Y\times G$,
for some holomorphic $h:Y\to G$. Since $\tilde{\psi}$ sends a section
of $p^{*}\left(E_{\sigma}\right)$ to a section of $p^{*}\left(E_{\rho}\right)$
we have $h(y)s_{\sigma}(y)=s_{\rho}(y)$, for all $y\in Y$, which
implies, $h(\gamma\cdot y)\sigma(\gamma)s_{\sigma}(y)=\rho(\gamma)h(y)s_{\sigma}(y)$,
as wanted.

$(2)\Leftrightarrow(3)$ Writing $h^{\gamma}(y):=h(\gamma\cdot y)$
we have, from (2), the equation $\rho(\gamma)=h^{\gamma}\sigma(\gamma)h^{-1}$,
for all $\gamma\in\pi_{1}$, as holomorphic functions on $Y$. Taking
the exterior derivative we get 
\[
0=d(h^{\gamma})\gamma'\sigma(\gamma)h^{-1}-h^{\gamma}\sigma(\gamma)h^{-1}dh\,h^{-1},
\]
($\gamma'$ is the derivative of $\gamma:y\mapsto\gamma y$) which
is equivalent to $(h^{\gamma})^{-1}dh^{y}\,\gamma'=\sigma(\gamma)h^{-1}dh\sigma(\gamma)^{-1}$,
for all $\gamma\in\pi_{1}$. Setting $\omega:=h^{-1}dh$, this equation
can be rewritten as $\Ad_{\sigma}(\gamma)\cdot\omega=\omega^{\gamma}\gamma'$,
which precisely means that the 1-form $\omega$ is the pullback to
$Y$ of a holomorphic section (still denoted the same) $\omega\in H^{0}(X,\Ad(E_{\sigma})\otimes\Omega_{X}^{1})$.
Conversely, since the solution of the differential equation $\omega=h^{-1}dh$
with the condition $h(y_{0})=e$ over the simply connected space $Y$
is unique and satisfies the equality $(3)$ then, obviously it satisfies
$(2)$.
\end{proof}
It is clear that Schottky space is different from the strict Schottky
space, when the center $Z$ is nontrivial. However, there is no need
to distinguish the strict and non-strict cases when considering their
associated bundles, in the case that $Z$ is itself connected, as
we now see.
\begin{prop}
\label{prop:produto-por-centro} Let $G$ be a complex connected reductive
group with center $Z$, and let $\rho:\pi_{1}\to G$ and $\sigma,\,\nu:\pi_{1}\to Z$
be representations. If there is an isomorphism $E_{\sigma}\cong E_{\nu}$
of $Z$-bundles, then the representations $\sigma\rho,\,\nu\rho\in\hom(\pi_{1},G)$,
give isomorphic $G$-bundles $E_{\sigma\rho}\cong E_{\nu\rho}$. \end{prop}
\begin{proof}
By Theorem \ref{analytic equivalence G bundles}, there exists a holomorphic
function $h:Y\rightarrow Z$ such that $\nu(\gamma)h(y)=h(\gamma\cdot y)\sigma(\gamma)$,
for every $\gamma\in\pi_{1}$, $y\in Y$. Considering this equation
in $G$, and since $\nu,\,\sigma$ are in the center of $G$, we can
multiply by $\rho(\gamma)$, obtaining: 
\[
\nu(\gamma)\rho(\gamma)h(y)=h(\gamma\cdot y)\sigma(\gamma)\rho(\gamma),\quad\quad\forall\gamma\in\pi_{1},\,y\in Y.
\]
Thus, $\nu\rho:\pi_{1}\to G$ is analytically equivalent to the Schottky
representation $\sigma\rho:\pi_{1}\to G$. So again by Theorem \ref{analytic equivalence G bundles},
$E_{\sigma\rho}\cong E_{\nu\rho}$. \end{proof}
\begin{prop}
\label{connected-center} Suppose that $Z$ is connected. Then $E$
is a $G$-Schottky bundle if and only if it is a strict $G$-Schottky
bundle. \end{prop}
\begin{proof}
A strict Schottky bundle is trivially a Schottky bundle. So, let $E=E_{\text{\ensuremath{\rho}}}$
be a Schottky $G$-bundle, with $\rho:\pi_{1}\to G$ a Schottky representation
and, using Theorem \ref{analytic equivalence G bundles}, we look
for a strict Schottky representation analytically equivalent to $\rho$.

Let $DG$ be the derived group of $G$. In terms of the well-known
decomposition $G=Z\cdot DG$, and our usual generators, we can write
$\rho(\alpha_{i})=\nu(\alpha_{i})\tilde{\rho}(\alpha_{i})$ and $\rho(\beta_{i})=\nu(\beta_{i})\tilde{\rho}(\beta_{i})$
for every $i=1,\cdots,g$, for some $\nu(\alpha_{i}),\nu(\beta_{i})\in Z$,
with $\tilde{\rho}(\beta_{i})\in DG$ and $\tilde{\rho}(\alpha_{i})=e$.
This assignment defines representations $\nu:\pi_{1}\to Z$ and $\tilde{\rho}:\pi_{1}\to DG$
satisfying $\rho(\gamma)=\nu(\gamma)\tilde{\rho}(\gamma)$ for all
$\gamma\in\pi_{1}$.

The representation $\nu$ defines a Schottky $Z$-bundle, $E_{\nu}$.
As $Z$ is connected, by Proposition \ref{prop: C star} there is
an isomorphism of $Z$-bundles $E_{\nu}\cong E_{\sigma}$, where $E_{\sigma}$
is the $Z$-bundle associated to a strict Schottky representation
$\sigma:\pi_{1}\to Z$, (so that $\sigma(\alpha_{i})=e$). By Proposition
\ref{prop:produto-por-centro}, $E_{\rho}=E_{\nu\tilde{\rho}}\cong E_{\sigma\tilde{\rho}}$.
Since $\sigma\tilde{\rho}:\pi_{1}\to G$ is a strict Schottky representation,
we are done. \end{proof}
\begin{example}
\label{exa:SL_n}Since $\mathbb{C}^{*}$, the center of $GL_{n}\mathbb{C}$,
is connected, every Schottky $GL_{n}\mathbb{C}$-bundle is strict
Schottky. On the other hand, for vector bundles with trivial determinant,
corresponding to $G=SL_{n}\mathbb{C}$, because $Z=\mathbb{Z}_{n}$,
our definition of Schottky bundles is more general than the one used
in \cite{Florentino}. 
\end{example}
For later use, we now provide another description of the fiber of
the uniformization map. 
\begin{defn}
\label{def Qrho}Given a representation $\rho\in\Hom\left(\pi_{1},\,G\right)$,
we define the following map, called the \emph{orbit map} 
\begin{eqnarray*}
Q_{\rho}:H^{0}(X,\Ad E_{\rho}\otimes\Omega_{X}^{1}) & \to & \mathbb{B}\\
\omega & \mapsto & Q_{\rho}\left(\omega\right):=\left[\sigma\right],
\end{eqnarray*}
with $\sigma\in\Hom\left(\pi_{1},\,G\right)$ the representation given
by 
\[
\sigma\left(\gamma\right):=h_{\omega}\left(\gamma\cdot y\right)\rho\left(\gamma\right)h_{\omega}\left(y\right)^{-1},\quad\gamma\in\pi_{1},\:y\in Y.
\]
Here, $h_{\omega}$ is defined in Theorem \ref{analytic equivalence G bundles}
(3), whose proof readily shows the following. \end{defn}
\begin{lem}
\label{lem:ImQrho equal fibre of E} The fibre $\mathbf{E}^{-1}(\left[E_{\rho}\right])$
coincides with $Q_{\rho}\left(H^{0}\left(X,\Ad E_{\rho}\otimes\Omega_{X}^{1}\right)\right)$,
the image of the orbit map. In other words, $E_{\rho}\cong E_{\sigma}$
if and only if $[\sigma]\in Im\left(Q_{\rho}\right)$. 
\end{lem}

\subsection{Tangent spaces and group cohomology\label{chap:Tangent-spaces}}

We now describe the tangent space of $\mathbb{B}=\Hom\left(\pi_{1},\:G\right)\quot G$,
at a good representation, in terms of the group cohomology of $\pi_{1}$.

More generally, let $\Gamma$ denote a finitely generated group and
fix $\rho\in\Hom\left(\Gamma,G\right)$. The adjoint representation
on the Lie algebra of $G$, $\mathfrak{g}=\mathsf{Lie}(G)$, composed
with $\rho$, that is 
\begin{equation}
\Ad_{\rho}:\Gamma\rightarrow G\to GL\left(\mathfrak{g}\right),\label{eq:ad-module}
\end{equation}
induces on $\mathfrak{g}$ a $\Gamma$-module structure, which we
denote by $\mathfrak{g}_{\mathrm{Ad}_{\rho}}$. The cohomology groups
of $\Gamma$ with coefficients in $\mathfrak{g}_{\Ad_{\rho}}$, are
explicitly given by: 
\[
H^{0}\left(\Gamma,\mathfrak{g}_{\Ad_{\rho}}\right):=Z^{0}\left(\Gamma,\mathfrak{g}_{\Ad_{\rho}}\right)=\left(\mathfrak{g}_{\Ad_{\rho}}\right)^{\Gamma}\quad\quad\mbox{(\ensuremath{\Gamma} invariants in \ensuremath{\mathfrak{g}_{\Ad_{\rho}}})},
\]
\[
H^{1}\left(\Gamma,\mathfrak{g}_{\Ad_{\rho}}\right):=Z^{1}\left(\Gamma,\mathfrak{g}_{\Ad_{\rho}}\right)/B^{1}\left(\Gamma,\mathfrak{g}_{\Ad_{\rho}}\right)
\]
where (see, e.g., \cite{Brown}) 
\begin{eqnarray*}
Z^{1}\left(\Gamma,\mathfrak{g}_{\Ad_{\rho}}\right) & := & \left\{ \left.\phi:\Gamma\to\mathfrak{g}\,\right|\phi(\gamma_{0}\gamma_{1})=\phi(\gamma_{0})+\Ad_{\rho}(\gamma_{0})\cdot\phi(\gamma_{1})\,,~\forall\gamma_{0},\gamma_{1}\in\Gamma\right\} ,\\
B^{1}\left(\Gamma,\mathfrak{g}_{\Ad_{\rho}}\right) & := & \left\{ \left.\phi:\Gamma\to\mathfrak{g}\,\right|\exists a\in\mathfrak{g},~~\phi(\gamma_{0})=\Ad_{\rho}(\gamma_{0})\cdot a-a\,,~\forall\gamma_{0}\in\Gamma\right\} .
\end{eqnarray*}
Let us recall the isomorphism between the Zariski tangent space of
the character variety at a good representation $\rho$, and the first
cohomology group $H^{1}\left(\Gamma,\mathfrak{g}_{\mathrm{Ad}_{\rho}}\right)$.
The following result was proved by Goldman \cite{GoldmanSymNatureFundGrp},
Martin \cite{Martin1} (generalizing the case of $G=GL_{n}\mathbb{C}$
proved by Weil \cite{Weil}) and Lubotzky and Magid \cite{LubotzkyMagid},
see also \cite{Sikora}.
\begin{thm}
\label{thm:tangent-space-group-cohom} For a good representation $\rho\in\Hom\left(\Gamma,G\right)$
we have, 
\[
T_{\left[\rho\right]}\left(\mathrm{Hom}\left(\Gamma,G\right)\quot G\right)\cong H^{1}\left(\Gamma,\mathfrak{g}_{\mathrm{Ad}_{\rho}}\right).
\]

\end{thm}
The identification between tangent spaces to character varieties and
group cohomology spaces is very useful in many situations. In particular,
we can use it to compute the dimension of the complex manifolds $\mathbb{B}^{\mathsf{gd}}=\Hom\left(\pi_{1},\:G\right)^{\mathsf{gd}}\quot G$
and $\mathbb{S}^{\mathsf{gd}}\subset\mathbb{B}^{\mathsf{gd}}$, consisting
of classes of good representations, when $\Gamma$ is the fundamental
group $\pi_{1}$ of a surface of genus $g$. In fact, by \cite[Lemma 6.2]{Martin1},
we have, for $\rho\in\mathbb{B}^{\mathsf{gd}}$: 
\[
\begin{array}{c}
\dim Z^{1}\left(\pi_{1},\mathfrak{g}_{\Ad_{\rho}}\right)=\left(2g-1\right)\dim G+\dim Z,\\
\dim B^{1}\left(\pi_{1},\mathfrak{g}_{\Ad_{\rho}}\right)=\dim G-\dim Z,
\end{array}
\]
and also the following. 
\begin{prop}
\cite{Martin1}\label{prop:dim and iso representations} If $[\rho]\in\mathbb{B}^{\mathsf{gd}}$,
then 
\[
T_{\left[\rho\right]}\mathbb{B}\cong H^{1}\left(\pi_{1},\mathfrak{g}_{\Ad_{\rho}}\right)
\]
and $\dim T_{\left[\rho\right]}\mathbb{B}=\left(2g-2\right)\dim G+2\dim Z.$ 
\end{prop}

\subsection{The period map}

As in Theorem \ref{analytic equivalence G bundles}, seing holomorphic
sections $\omega\in H^{0}(X,\Ad E_{\rho}\otimes\Omega_{X}^{1})$ as
1-forms on the universal cover $Y$, we can integrate them along paths
to obtain elements in group cohomology of $\pi_{1}$. This defines
the period map.

Fix $y\in Y$ and $\omega\in H^{0}(X,\Ad E_{\rho}\otimes\Omega_{X}^{1})$.
Let us denote by $\phi_{y}^{\omega}$ the map: 
\[
\begin{array}{ccc}
\phi_{y}^{\omega}:\,\pi_{1} & \rightarrow & \mathfrak{g}\\
\gamma & \mapsto & \phi_{y}^{\omega}(\gamma):=\int_{y}^{\gamma\cdot y}\omega,
\end{array}
\]
where we denote also by $\omega$ its pullback to $Y$. In fact, $\phi_{y}^{\omega}$
is a cocycle in $Z^{1}(\pi_{1},\mathfrak{g}_{\Ad_{\rho}})$, and its
cohomology class only depends on $\omega$ (and not on the basepoint
$y\in Y$). 
\begin{prop}
\label{prop:Period map-1} Fix a representation $\rho:\pi_{1}\to G$,
and $y\in Y$. Then, for every $\omega$, $\phi_{y}^{\omega}\in Z^{1}(\pi_{1},\mathfrak{g}_{\Ad_{\rho}})$.
Moreover, the assignement 
\begin{eqnarray*}
P_{\Ad_{\rho}}:H^{0}\left(X,\Ad E_{\rho}\otimes\Omega_{X}^{1}\right) & \rightarrow & H^{1}\left(\pi_{1},\mathfrak{g}_{\Ad_{\rho}}\right)\\
\omega & \mapsto & [\phi_{y}^{\omega}],
\end{eqnarray*}
is a well defined linear map between finite dimensional $\mathbb{C}$-vector
spaces, and is independent of $y\in Y$. \end{prop}
\begin{defn}
We call $P_{\Ad\rho}$, as defined above, the \emph{period map} associated
with $\rho$.\end{defn}
\begin{proof}
Using the action on 1-forms $\gamma\cdot\omega=\Ad_{\rho}(\gamma)\cdot\omega=(\omega\circ\gamma)\gamma'$,
$\gamma\in\pi_{1}$, as in Theorem \ref{analytic equivalence G bundles},
we compute, by linearity and change of variable:
\begin{eqnarray*}
\phi_{y}^{\omega}\left(\gamma_{1}\gamma_{2}\right) & = & {\textstyle \int_{y}^{\gamma_{1}\cdot y}\omega+\int_{\gamma_{1}\cdot y}^{\gamma_{1}\cdot(\gamma_{2}\cdot y)}\omega}\\
 & = & {\textstyle \phi_{y}^{\omega}(\gamma_{1})+\int_{y}^{\gamma_{2}\cdot y}(\omega\circ\gamma_{1})\,\gamma_{1}'}\\
 & = & \phi_{y}^{\omega}(\gamma_{1})+\gamma_{1}\cdot\phi_{y}^{\omega}(\gamma_{2}),
\end{eqnarray*}
which shows that $\phi_{y}^{\omega}$ is a cocycle in $Z^{1}\left(\pi_{1},\mathfrak{g}_{\Ad_{\rho}}\right)$.
The proof that $P_{\Ad_{\rho}}(\omega)$ is independent of the base
point follows a similar computation to conclude that $\phi_{y}^{\omega}-\phi_{y'}^{\omega}$
is $1$-coboundary, for another $y'\in Y$. 
\end{proof}
Recall that the orbit map 
\[
Q_{\rho}:\ H^{0}(X,\Ad E_{\rho}\otimes\Omega_{X}^{1})\to\mathbb{B},
\]
(see Definition \ref{def Qrho}) verifies $Q_{\rho}(0)=[\rho]$, and
its derivative at the identity, for a good representation $\rho$
is a map: 
\[
d_{0}Q_{\rho}:\,H^{0}(X,\Ad E_{\rho}\otimes\Omega_{X}^{1})\to T_{[\rho]}\mathbb{B}\cong H^{1}(\pi_{1},\,\mathfrak{g}_{\Ad_{\rho}}).
\]

\begin{lem}
\label{imdQequaldE}For a representation $\rho\in\Hom\left(\pi_{1},G\right)$,
the image of $d_{0}Q_{\rho}$ coincides with the kernel of the map
$d_{[\rho]}\mathbf{E}$, the derivative of $\mathbf{E}$ at $[\rho]$. \end{lem}
\begin{proof}
Using Theorem \ref{analytic equivalence G bundles}, and Lemma \ref{lem:ImQrho equal fibre of E},
the proof is analogous to the proof of \cite[Lemma 4(a)]{Florentino}. 
\end{proof}
For a good representation $\rho\in\Hom\left(\pi_{1},G\right)$, such
that $[\rho]\in\mathbf{E}^{-1}(\mathcal{M}_{G})$, we can form the
diagram 
\begin{equation}
\xymatrix{H^{0}\left(X,\Ad E_{\rho}\otimes\Omega_{X}^{1}\right)\ar@{->}[r]^{~~~~~~~~d_{0}Q_{\rho}}\ar@{->}[rd]_{~~~~~~~~P_{\Ad_{\rho}}} & T_{[\rho]}\mathbb{B}\ar@{->}[d]^{\cong}\ar@{->}[r]^{d_{[\rho]}\mathbf{E}} & T_{[E_{\rho}]}\mathcal{M}_{G}\\
 & H^{1}\left(\pi_{1},\mathfrak{g}_{\Ad_{\rho}}\right).
}
\label{eq:dQrho=00003D00003D00003DPAd}
\end{equation}
The next result shows that, in fact, the triangle above is commutative. 
\begin{prop}
\label{PeriodMapEQdiFQrho} For each good representation $\rho$ in
$\Hom\left(\pi_{1},G\right)$, the maps $d_{0}Q_{\rho}$ and $P_{\Ad_{\rho}}$,
coincide under the vertical isomorphism of diagram \eqref{eq:dQrho=00003D00003D00003DPAd}. \end{prop}
\begin{proof}
Since $G$ is a connected reductive group over complex numbers, there
is a faithful representation $\phi:G\to GL_{n}\mathbb{C}$. By associating
a representation $\rho\in\Hom\left(\pi_{1},\,G\right)$ to the composition
$\phi\circ\rho=\bar{\rho}\in\Hom\left(\pi_{1},\,GL_{n}\mathbb{C}\right)$,
$\phi$ induces an injective morphism of algebraic varieties $\bar{\phi}:\mathbb{B}\to\mathbb{G}_{n}$,
where $\mathbb{G}_{n}=\Hom\left(\pi_{1},\,GL_{n}\mathbb{C}\right)\quot GL_{n}\mathbb{C}$.
The Lie algebra $\mathfrak{g}$, can be seen as a subalgebra of the
Lie algebra $\mathfrak{gl}=M_{n}\mathbb{C}$ of $GL_{n}\mathbb{C}$,
and we obtain an inclusion of $\pi_{1}$-modules $\mathfrak{g}_{\mathrm{Ad}_{\rho}}\subset\mathfrak{gl}_{\mathrm{Ad}_{\bar{\rho}}}$.
On the other hand, Florentino proved in \cite[Lemma 4(b)]{Florentino}
this result for $G=GL_{n}\mathbb{C}$. So, we obtain the following
diagram, where $E_{\bar{\rho}}$ is the associated vector bundle of
$E_{\rho}$.{\small\[\xymatrixcolsep{0.2in}\xymatrixcolsep{0.2in}\xymatrix@C-=0.3cm{H^1\left(\pi_{1},\, \mathfrak{g}_{\mathrm{Ad}_\rho}\right)\ar@{->}[rrrr]^(.3){}&&&&H^1\left(\pi_{1},\, \mathfrak{gl}_{\mathrm{Ad}_{\bar{\rho}}}\right) \\ &T_{\left[\rho\right]}\mathbb{B}\ar@{->}[lu]^-{\cong}\ar@{->}[rr]^(1){d_{\left[\rho\right]}\bar{\phi}~~~~~~~~~~~~~~~~~~~}&&T_{\left[\bar{\rho}\right]}\mathbb{G}_{n}\ar@{->}[ru]^-{\cong} & \\ H^0\left(X, \Ad(E_{\rho})\otimes \Omega^1_X\right)\ar@{->}[rrrr]^(.2){}\ar@{->}[uu]^{P_{\mathrm{Ad}_{\rho}}}\ar@{->}[ru]^-{d_0Q_{\rho}}&&&&H^0\left(X, \mathrm{End}\left(E_{\bar{\rho}}\right)\otimes \Omega^1_X\right)\ar@{->}[lu]_-{d_0Q_{\bar{\rho}} }\ar@{->}[uu]^{P_{\mathrm{Ad}_{\bar{\rho}}}}}\] }Above,
the horizontal arrows are inclusions of vector spaces, because $H^{0}$
and $H^{1}$ behave functorially. Finally, since the triangle on the
right is commutative, the same holds for the left triangle, as wanted. 
\end{proof}

\section{Schottky space\label{sec:Schottky-moduli-map}}

In this section we compute the dimension of Schottky space and prove
that the strict Schottky space is a Lagrangian subspace of the Betti
space. We also define the Schottky uniformization and moduli maps,
by restricting the uniformization map to Schottky representations,
and to those representations whose flat bundles are semistable.

\subsection{Dimension of Schottky space}

We now compute the dimensions of $\mathbb{S}$ and $\mathbb{S}_{s}$,
using the techniques of group cohomology. By the density result (Theorem
\ref{thm:good rep for g2}), the computations can be carried out at
good representations. Using formula \eqref{eq:S=00003DZvezesS_s}
we can write the inclusion $\mathbb{S}^{\mathsf{gd}}\subset\mathbb{B}^{\mathsf{gd}}$
as 
\begin{eqnarray*}
\hom(F_{g},Z)\times\hom(F_{g},G)^{\mathsf{gd}}\quot G\cong Z^{g}\times\mathbb{S}_{s}^{\mathsf{gd}} & \hookrightarrow & \mathbb{B}^{\mathsf{gd}}\cong\hom(\pi_{1},G)^{\mathsf{gd}}\quot G\\
(\rho_{1},[\rho_{2}]) & \mapsto & [\rho].
\end{eqnarray*}
Above, the notation should be clear according to Section \ref{sec:Schottky-Representations}.
Correspondingly, from Theorem \ref{thm:tangent-space-group-cohom},
we obtain the inclusion of tangent spaces: 
\begin{equation}
T_{[\rho]}\mathbb{S}=T_{\rho_{1}}(Z^{g})\oplus T_{[\rho_{2}]}\mathbb{S}_{s}\cong\mathfrak{z}^{g}\oplus H^{1}\left(F_{g},\mathfrak{g}_{\Ad_{\rho_{2}}}\right)\hookrightarrow H^{1}\left(\pi_{1},\mathfrak{g}_{\Ad_{\rho}}\right)=T_{[\rho]}\mathbb{B}\label{eq:goodtangentspaces}
\end{equation}
for $[\rho]\in\mathbb{S}^{\mathsf{gd}}$. Recall that $\mathfrak{g}\mathfrak{_{\mathrm{Ad}_{\rho_{2}}}}$
denotes the $F_{g}$-module $\lie(G)=\mathfrak{g}$, with the $F_{g}$-action
given by the composition $F_{g}\stackrel{\rho_{2}}{\to}G\stackrel{\mathrm{Ad}}{\to}GL(\mathfrak{g})$. 
\begin{prop}
\label{Thm strict Schotky fst coh grp}Let $g\geq2$. We have $\dim\mathbb{S}_{s}=\left(g-1\right)\dim G+\dim Z.$ \end{prop}
\begin{proof}
Since good representations are dense in $\mathbb{S}_{s}$, it is enough
to compute the dimension at a strict good representation, $[\rho]\in\mathbb{S}_{s}^{\mathsf{gd}}$,
$\rho:F_{g}\to G$. By Theorem \ref{thm:tangent-space-group-cohom},
we know 
\[
\dim\mathbb{S}_{s}=\dim T_{[\rho]}\mathbb{S}_{s}=\dim H^{1}\left(F_{g},\mathfrak{g_{\mathrm{Ad_{\rho}}}}\right).
\]
Since $F_{g}$ is a free group, there is no cocycle condition, so
any 1-cocycle is completely defined by the image of its generators;
this means that $Z^{1}(F_{g},\mathfrak{g_{\mathrm{Ad_{\rho}}}})\cong\mathfrak{g}^{g}.$
In order to compute the dimension of the space of $1$-coboundaries,
$B^{1}\left(F_{g},\mathfrak{g_{\mathrm{Ad_{\rho}}}}\right)$, we consider
the linear map between vector spaces 
\[
\begin{array}{cccl}
\psi_{\rho}: & \mathfrak{g} & \to & \mathfrak{g}^{g}\\
 & v & \mapsto & \left(\rho\left(\gamma_{1}\right)v\rho(\gamma_{1})^{-1}-v,\cdots,\rho\left(\gamma_{g}\right)v\rho(\gamma_{g})^{-1}-v\right),
\end{array}
\]
and note that $B^{1}(F_{g},\mathfrak{g_{\mathrm{Ad_{\rho}}}})=\psi_{\rho}(\mathfrak{g})$.
Thus: 
\[
\begin{array}{c}
\dim B^{1}(F_{g},\mathfrak{g_{\mathrm{Ad_{\rho}}}})=\dim\psi_{\rho}(\mathfrak{g})=\dim\mathfrak{g}-\dim\ker\psi_{\rho}=\dim\mathfrak{g}-\dim\mathfrak{z}(\rho)\end{array}
\]
where $\mathfrak{z}(\rho):=\left\{ v\in\mathfrak{g}|\:v\rho(\gamma_{i})=\rho(\gamma_{i})v,\forall i=1,\cdots,g\right\} $
is the Lie algebra of the stabilizer of $\rho$, $Z(\rho)$. Finally,
\[
\dim H^{1}\left(F_{g},\mathfrak{g_{\mathrm{Ad_{\rho}}}}\right)=\dim Z^{1}\left(F_{g},\mathfrak{g_{\mathrm{Ad_{\rho}}}}\right)-\dim B^{1}\left(F_{g},\mathfrak{g_{\mathrm{Ad_{\rho}}}}\right)=g\dim G-\dim G+\dim Z(\rho).
\]
Since $\rho$ is good, by definition $Z(\rho)=Z$, and the proof is
finished. \end{proof}
\begin{cor}
\label{cor: dimension and iso of smooth Schottky rep } For $g\geq2$,
the Schottky space $\mathbb{S}$ is equidimensional (all irreducible
components have the same dimension). Moreover, 
\[
\dim\mathbb{S}=\left(g-1\right)\dim G+\left(g+1\right)\dim Z.
\]
\end{cor}
\begin{proof}
This follows immediately from the previous result and from Proposition
\ref{Prop:many-irred-components}, as $\dim Z^{\circ}=\dim Z$. 
\end{proof}

\subsection{Lagrangian subspaces of $\mathbb{S}_{s}$.}

Recall that a Lagrangian submanifold $L\subset M$ of a symplectic
manifold $M$ is a half dimensional submanifold such that the symplectic
form vanishes on any tangent vectors to $L$.

It is well known that character varieties of surface group representations
have a natural symplectic structure (\cite{GoldmanSymNatureFundGrp}),
which can be constructed as follows. Consider an Ad-invariant bilinear
form $\left\langle \,,\,\right\rangle $ on $\mathfrak{g}$. Then,
using the cup product on group cohomology 
\begin{equation}
\cup:H^{1}\left(\pi_{1},\mathfrak{g}_{\Ad_{\rho}}\right)\otimes H^{1}\left(\pi_{1},\mathfrak{g}_{\Ad_{\rho}}\right)\to H^{2}\left(\pi_{1},\mathfrak{g}_{\Ad_{\rho}}\right),\label{eq:cup-product}
\end{equation}
and composing it with the contraction with $\left\langle \,,\,\right\rangle $
and with the evaluation on the fundamental 2-cycle, we obtain a non-degenerate
bilinear pairing: 
\begin{equation}
H^{1}\left(\pi_{1},\mathfrak{g}_{\Ad_{\rho}}\right)\otimes H^{1}\left(\pi_{1},\mathfrak{g}_{\Ad_{\rho}}\right)\overset{\cup}{\longrightarrow}H^{2}\left(\pi_{1},\,\mathfrak{g}_{\Ad_{\rho}}\right)\overset{\langle\,,\,\rangle}{\longrightarrow}H^{2}\left(\pi_{1},\,\mathbb{C}\right)\cong\mathbb{C}\label{eq:cup product}
\end{equation}
Under the identification of $H^{1}\left(\pi_{1},\mathfrak{g}_{\Ad_{\rho}}\right)$
with the tangent space at a good representation $\rho\in\mathbb{B}^{\mathsf{gd}}$,
this pairing defines a complex sympletic form on the complex manifold
$\mathbb{B}^{\mathsf{gd}}$. This symplectic form is complex analytic
with respect to the complex structure on $\mathbb{B}^{\mathsf{gd}}$
coming from the complex structure on $G$, and $\mathbb{S}_{s}^{\mathsf{gd}}\subset\mathbb{B}^{\mathsf{gd}}$
is Lagrangian.\footnote{For a general real Lie group, the analogous pairing defines a smooth
($C^{\infty}$) symplectic structure, see \cite{GoldmanSymNatureFundGrp}.}
\begin{prop}
\label{prop:Schottky-lagrangian}The good locus of the strict Schottky
space $\mathbb{S}_{s}^{\mathsf{gd}}$ is a Lagrangian submanifold
of $\mathbb{B}^{\mathsf{gd}}$. \end{prop}
\begin{proof}
The restriction of the map \eqref{eq:cup-product} to $H^{1}\left(F_{g},\mathfrak{g}_{\Ad_{\rho}}\right)$
is a vanishing map: 
\[
\cup:H^{1}\left(F_{g},\mathfrak{g}_{\Ad_{\rho}}\right)\otimes H^{1}\left(F_{g},\mathfrak{g}_{\Ad_{\rho}}\right)\to H^{2}\left(F_{g},\mathfrak{g}_{\Ad_{\rho}}\right)=0,
\]
because free groups have vanishing higher cohomology groups (see \cite{Brown}).
Since the tangent space, at a good point, to the strict Schottky locus
$\mathbb{S}_{s}$ is identified with $H^{1}\left(F_{g},\mathfrak{g}_{\Ad_{\rho}}\right)$
(see Theorem \ref{thm:tangent-space-group-cohom}), this means that
the symplectic form, defined above on $\mathbb{B}^{\mathsf{gd}}$,
vanishes on any two tangent vectors to $\mathbb{S}_{s}^{\mathsf{gd}}$.
Since the dimension of $\mathbb{B}^{\mathsf{gd}}$ is twice the dimension
of $\mathbb{S}_{s}^{\mathsf{gd}}$ (see Proposition \ref{prop:dim and iso representations}
and Proposition \ref{Thm strict Schotky fst coh grp}), we conclude
the result. \end{proof}
\begin{rem}
(1) The proof that $\mathcal{L}_{G}$ is Lagrangian is only done for
complex semisimple groups in \cite{BaragliaSchaposnick}. Thus, Proposition
\ref{prop:Schottky-brane} generalizes that statement for reductive
complex algebraic groups. Moreover, since there are good strict Schottky
representations for every $g\geq2$, the current approach furnishes
a proof that the Baraglia-Schaposnik branes are non-empty, at least
in the conditions of Remark \ref{rem:g+1-loops}.\\
 (2) Proposition \ref{prop:Schottky-brane} shows that we have an
inclusion $\mathbb{S}_{s}\subset\mathcal{L}_{G}$ in the $(A,B,A)$-branes
of \cite{BaragliaSchaposnick} and in the case $G$ is an adjoint
group, $\mathbb{S}=\mathbb{S}_{s}\subset\mathcal{L}_{G}$. In a future
work, we plan to study the conditions under which this inclusion is
actually a bijection.%

\end{rem}

\subsection{The Schottky uniformization and moduli maps}
\begin{defn}
The \emph{Schottky uniformization map} 
\begin{equation}
\mathbf{W}:\,\mathbb{S}\to M_{G}\label{eq: Schottky moduli map on the cat quot-1-1}
\end{equation}
is defined by $\mathbf{W}[\rho]:=[E_{\rho}]$, the isomorphism class
of the Schottky $G$-bundle $E_{\rho}$. From \eqref{eq:uniformization-map},
$\mathbf{W}=\mathbf{E}\circ i$ where $i:\mathbb{S}\to\mathbb{B}$
is the inclusion from Proposition \ref{prop:SchottkyAndFreeGroup}. \end{defn}
\begin{rem}
\label{rem:nosemistable}(1) As mentioned above, $\mathbf{W}[\rho]$
is not necessarily semistable. In fact, maximally unstable rank vector
2 bundles with trivial determinant are Schottky (see \cite{Florentino}).
Also, $\mathbb{\mathbf{W}}$ is not injective in general: this happens
already for the line bundle case (see \cite{Florentino}).\\
 (2) Recall that, from Theorem \ref{thm:Schottky bdl has trivial type},
$\mathbf{W}[\rho]$ has trivial topological type.%

\end{rem}
As defined, the target of the Schottky uniformization map $M_{G}$
is a set, and it can be given the structure of a stack. However, since
we want to consider the relation between Schottky space $\mathbb{S}$
and the moduli space of $G$-bundles, we need to further restrict
$\mathbb{\mathbf{W}}$ to be a morphism of algebraic varieties.

Let $\mathcal{M}_{G}^{F}=\mathcal{M}_{G}\cap M_{G}$ be the moduli
space of semistable $G$-bundles on $X$ that admit a flat connection.
It is a, generally singular, projective complex algebraic variety.
In order to characterize the derivative of the Schottky map $\mathbf{W}$,
we will consider the subsets 
\[
\mathbb{B}^{*}:=\mathbf{E}^{-1}\left(\mathcal{M}_{G}^{F}\right),\quad\quad\mathbb{S}^{*}:=\mathbf{W}^{-1}\left(\mathcal{M}_{G}^{F}\right),
\]
consisting of representations (resp. Schottky representations) $[\rho]$
whose associated bundles $E_{\rho}$ are semistable. 
\begin{prop}
For $g\geq2$, the subset $\mathbb{S}^{*}\subset\mathbb{S}$ contains
the unitary Schottky representations. Moreover, $\mathbb{S}^{*}\cap\mathbb{S}^{\mathsf{gd}}$
is open in $\mathbb{S}$. \end{prop}
\begin{proof}
By Proposition \ref{prop: existence of good and unitary Schottky rep }
and Theorem \ref{thm:good rep for g2} we know that $\mathbb{S}^{\mathsf{gd}}$
contains unitary representations and it is smooth and open in $\mathbb{S}$,
since $g\geq2$. If $\rho\in\mathcal{S}$ is a unitary representation,
then $E_{\rho}$ is semistable by Ramanathan's theorem. So, $[E_{\rho}]\in\mathcal{M}_{G}^{F}$
and $[\rho]\in\mathbf{W}^{-1}(E_{\rho})\subset\mathbb{S}^{*}$. Thus
$\mathbb{S}^{*}\cap\mathbb{S}^{\mathsf{gd}}$ is non-empty, so it
is open in $\mathbb{S}$, by the coarse moduli property.%
\end{proof}
\begin{defn}
\label{def:Schottky-moduli-map}The \emph{Schottky moduli map} 
\begin{equation}
\mathbf{V}:\,\mathbb{S}^{*}\to\mathcal{M}_{G}\label{eq: Schottky moduli map on the cat quot-1-1-1}
\end{equation}
is defined to be the restriction of the Schottky uniformization map
$\mathbf{W}$ to the subset $\mathbb{S}^{*}=\mathbf{W}^{-1}(\mathcal{M}_{G}^{F})\subset\mathbb{S}$
of representations defining semistable $G$-bundles. %
\end{defn}
\begin{thm}
\label{thm:kernel of Schottky map} Let $\rho$ be a good Schottky
representation, then 
\begin{equation}
\ker d_{[\rho]}\mathbf{V}\cong T_{[\rho]}\mathbb{S}\bigcap\mathrm{Im}\,d_{0}Q_{\rho}.\label{eq: ker dWrho}
\end{equation}
\end{thm}
\begin{proof}
It is immediate from Lemma \ref{imdQequaldE}. 
\end{proof}

\section{Surjectivity of the Schottky moduli map\label{sec:Surjectivity-of-the}}

In this section, we consider the image of the Schottky moduli map
inside the moduli space of semistable $G$-bundles. The main result
is the proof that this map is a local submersion at a good and unitary
Schottky representation (see Theorem \ref{thm: main thm}).

\subsection{Bilinear relations\label{chap:Schottky moduli map}}

Let again $K$ denote a maximal compact subgroup of the complex connected
reductive algebraic group $G$. We fix an hermitian structure on the
complex Lie algebra $\mathfrak{g}$ of $G$, denoted by $\left\langle \,,\,\right\rangle :\mathfrak{g}\times\mathfrak{g}\to\mathbb{C}$
($\mathbb{C}$-linear on the first entry) which is invariant under
the adjoint action of $K$ on $\mathfrak{g}$. For example, if $G=GL_{n}\mathbb{C}$,
we can take $\left\langle A,\,B\right\rangle :=\mathrm{tr}\left(AB^{*}\right),$
$\forall A,\,B\in\mathfrak{g}$ ,where $*$ means conjugate transpose
and $\mbox{tr}$ the matrix trace.

We now define an hermitian inner product on $H^{0}\left(X,\Ad\left(E_{\rho}\right)\otimes\Omega_{X}^{1}\right)$,
when $\rho:\pi_{1}\to K\subset G$ is a unitary representation. As
before, $Y$ is a universal cover of the compact Riemann surface $X$
of genus $g\geq2$, and we let $D\subset Y$ denote a fundamental
domain for the quotient $X=\left.Y\right/\pi_{1}$. %

\begin{defn}
\label{def 1st def hermitian inner prod}Let $\omega_{1},\,\omega_{2}\in H^{0}\left(X,\Ad E_{\rho}\otimes\Omega_{X}^{1}\right)$,
with $\rho:\pi_{1}\to K\subset G$. Define the following hermitian
inner product 
\begin{equation}
\left(\omega_{1},\,\omega_{2}\right):=i\int_{X}\left\langle \omega_{1},\omega_{2}\right\rangle :=i\int_{D}\left\langle h_{1}(z),\,h_{2}(z)\right\rangle dz\wedge d\bar{z}\label{eq: inner product 1forms}
\end{equation}
where $\omega_{i}=h_{i}(z)dz$ for $z\in Y$. \end{defn}
\begin{rem*}
The above integral depends on the choice of the hermitian inner product
on $\mathfrak{g}$. However, by unitarity of $\rho$, it is independent
of the choice of the fundamental domain $D$.%

\end{rem*}
To prove the unitarity of the period map $\omega\mapsto P_{\Ad_{\rho}}\left(\omega\right)$,
at unitary representations, generalizing \cite[Proposition 5]{Florentino},
we need also a pairing on $H^{1}\left(\pi_{1},\,\mathfrak{g}_{\Ad_{\rho}}\right)$.
We use the so-called Fox calculus, and extend 1-cocycles $\phi:\pi_{1}\to\mathfrak{g}_{\Ad_{\rho}}$
by $\mathbb{Z}$-linearity to the group ring $\mathbb{Z}\left[\pi_{1}\right]$
(see \cite{Florentino,GoldmanSymNatureFundGrp}). The boundary $\partial D$
can be considered as a $4g$ polygon, with a vertex $z_{0}\in Y$,
and the other vertices ordered as: 
\[
\left\{ z_{0},\,\alpha_{1}z_{0},\,\alpha_{1}\beta_{1}z_{0},\,\alpha_{1}\beta_{1}\alpha_{1}^{-1}z_{0},\,R_{1}z_{0},\,R_{1}\alpha_{2}z_{0},\cdots,\,R_{g}z_{0}=z_{0}\right\} 
\]
where $R_{k}=\prod_{i=1}^{k}\alpha_{i}\beta_{i}\alpha_{i}^{-1}\beta_{i}^{-1}$,
and define $R:=R_{g}$. The Fox derivatives of $R$ give: 
\[
\begin{array}{ccc}
{\displaystyle \frac{\partial R}{\partial\alpha_{i}}}:=R_{i-1}-R_{i}\beta_{i}, & \quad\quad\quad & {\displaystyle \frac{\partial R}{\partial\beta_{i}}}:=R_{i-1}\alpha_{i}-R_{i}\end{array}.
\]
Introduce also a $\mathbb{Z}$-linear involution $\sharp$ on $\mathbb{Z}\left[\pi_{1}\right]$
defined by $\sharp\left(\sum n_{i}\gamma_{i}\right):=\sum n_{i}\gamma_{i}^{-1}$,
$n_{i}\in\mathbb{Z}$. In particular: 
\begin{equation}
\begin{array}{ccc}
\sharp{\displaystyle \frac{\partial R}{\partial\alpha_{i}}}=R_{i-1}^{-1}-\beta_{i}^{-1}R_{i}^{-1} & \,\text{\,and}\,\:\: & \sharp{\displaystyle \frac{\partial R}{\partial\beta_{i}}}=\alpha_{i}^{-1}R_{i-1}^{-1}-R_{i}^{-1}.\end{array}\label{eq:involution Ri}
\end{equation}

\begin{defn}
Define a pairing on $H^{1}\left(\pi_{1},\mathfrak{g}_{\Ad_{\rho}}\right)$
by 
\[
\left\langle \!\left\langle \phi_{1},\,\phi_{2}\right\rangle \!\right\rangle :=i{\displaystyle \ \sum_{j=1}^{g}}\left\langle \phi_{1}\left(\sharp{\textstyle \frac{\partial R}{\partial\beta_{j}}}\right),\,\phi_{2}\left(\beta_{j}\right)\right\rangle -\left\langle \phi_{1}\left(\sharp{\textstyle \frac{\partial R}{\partial\alpha_{j}}}\right),\,\phi_{2}\left(\alpha_{j}\right)\right\rangle \!\!,
\]
for $\phi_{1},\,\phi_{2}\in Z^{1}\left(\pi_{1},\mathfrak{g}_{\Ad_{\rho}}\right)$. \end{defn}
\begin{rem}
It can be shown that this pairing is well defined on cohomology classes,
and is hermitian (being a complex analogue of the pairing in \cite{GoldmanSymNatureFundGrp}).
Moreover, it coincides (up to the factor $i$) with the cup product
pairing in \eqref{eq:cup-product}, when using our hermitian structure
on $\mathfrak{g}$.\end{rem}
\begin{thm}
\label{prop: Bilinear relations for differentials of pi1module} Let
$\rho:\pi_{1}\to K\subset G$ be a unitary and good representation.
Then, for all $\omega_{1},\,\omega_{2}\in H^{0}\left(X,\Ad\left(E_{\rho}\right)\otimes\Omega_{X}^{1}\right)$,
we have: 
\[
\left(\omega_{1},\omega_{2}\right)=\left\langle \!\left\langle P_{\Ad_{\rho}}\left(\omega_{1}\right),\,P_{\Ad_{\rho}}\left(\omega_{2}\right)\right\rangle \!\right\rangle .
\]
In other words, at a good and unitary representation, the period map
is unitary. \end{thm}
\begin{proof}
Fix a base point $y=z_{0}\in Y$, let $\phi_{1}\left(\gamma\right):=\int_{y}^{\gamma\cdot y}\omega_{1}$
be a 1-cocycle representing $P_{\Ad_{\rho}}\left(\omega_{1}\right)$,
and write $\omega_{2}=h(z)\,dz$. Define also $f:Y\to\mathfrak{g}$
by $f\left(z\right):=\int_{y}^{z}\omega_{1}$, so that we have $\omega_{1}=df$
(with a slight abuse of notation we identify forms in $D\subset X$
with their pullbacks to $Y$). Computing as in Proposition \ref{prop:Period map-1},
this function verifies 
\begin{equation}
f(\gamma z)=\phi_{1}(\gamma)+\int_{\gamma\cdot y}^{\gamma\cdot z}\omega_{1}=\phi_{1}(\gamma)+\int_{y}^{z}\gamma\cdot\omega=\phi_{1}(\gamma)+\gamma\cdot f(z),\label{eq: f(gamma) property}
\end{equation}
where we write $\gamma\cdot f$ for the $\Ad_{\rho}$-action of $\pi_{1}$
on functions on $Y$. Note that, for 1-forms on $Y$, we have $\gamma\cdot h\,dz=h(\gamma z)\gamma'(z)\,dz$.

Using $\left\langle \omega_{1},\omega_{2}\right\rangle =d\left\langle f,\,hdz\right\rangle =d\left(\left\langle f,\,h\right\rangle \overline{dz}\right)$,
applying Stokes' theorem to \eqref{eq: inner product 1forms}, and
decomposing the boundary $\partial D$ as the $4g$ polygon described
above, we get: 
\begin{align}
\left(\omega_{1},\,\omega_{2}\right) & =i\int_{\partial D}\left\langle f\left(z\right),h\left(z\right)\right\rangle \overline{dz}=\label{eq:inner product 2}\\
 & =i{\displaystyle \int}_{\!\!\!y}^{\alpha_{1}y}\left\langle f(z),h(z)\right\rangle \,\,\overline{dz}+\cdots+i{\displaystyle \int}_{\!\!\!R_{g-1}\alpha_{g}\beta_{g}\alpha_{g}^{-1}y}^{R_{g}y}\left\langle f(z),h(z)\right\rangle \,\overline{dz}\nonumber 
\end{align}
For each $j=1,\cdots,g$, we reduce the pair of integrals: 
\begin{equation}
\int_{R_{j-1}y}^{R_{j-1}\alpha_{j}y}\!\!\!\left\langle f,h\right\rangle \overline{dz}+\int_{R_{j-1}\alpha_{j}\beta_{j}y}^{R_{j-1}\alpha_{j}\beta_{j}\alpha_{j}^{-1}y}\!\!\!\left\langle f,h\right\rangle \overline{dz},\label{eq:integrals-alpha}
\end{equation}
to a single one by using the change of variables property \eqref{eq: f(gamma) property},
and the $\Ad_{\rho}$-invariance $\left\langle \gamma\cdot f,\,\gamma\cdot h\right\rangle =\left\langle f,\,h\right\rangle $
for all $\gamma\in\pi_{1}$. Employing the notation $f^{\gamma}\equiv f\circ\gamma$,
and using $R_{j-1}\alpha_{j}\beta_{j}\alpha_{j}^{-1}=R_{j}\beta_{j}$,
the expression \eqref{eq:integrals-alpha} equals: 
\begin{align*}
 & \int_{R_{j-1}y}^{R_{j-1}\alpha_{j}y}\left\langle f,h\right\rangle \,\overline{dz}-\int_{R_{j}\beta_{j}y}^{R_{j}\beta_{j}\alpha_{j}y}\left\langle f,h\right\rangle \,\overline{dz}\ =\\
 & =\int_{y}^{\alpha_{j}y}\left(\left\langle f^{R_{j-1}},h^{R_{j-1}}\right\rangle \overline{R_{j-1}'}-\left\langle f^{R_{j}\beta_{j}},h^{R_{j}\beta_{j}}\right\rangle \overline{(R_{j}\beta_{j})'}\right)\,\overline{dz}\\
 & =\int_{y}^{\alpha_{j}y}\!\!\!\left(\left\langle \phi_{1}(R_{j-1})+R_{j-1}\!\cdot\!f,\,R_{j-1}\!\cdot\!h\right\rangle -\left\langle \phi_{1}(R_{j}\beta_{j})+R_{j}\beta_{j}\!\cdot\!f,\,R_{j}\beta_{j}\!\cdot\!h\right\rangle \right)\,\overline{dz}\\
 & =\int_{y}^{\alpha_{j}y}\left(\left\langle \phi_{1}(R_{j-1}),\,R_{j-1}\cdot h\right\rangle -\left\langle \phi_{1}(R_{j}\beta_{j}),\,R_{j}\beta_{j}\cdot h\right\rangle \right)\,\overline{dz}\\
 & =\int_{y}^{\alpha_{j}y}\left(-\left\langle \phi_{1}(R_{j-1}^{-1}),\,h\right\rangle +\left\langle \phi_{1}(\beta_{j}^{-1}R_{j}^{-1}),\,h\right\rangle \right)\,\overline{dz}\\
 & =-\left\langle \phi_{1}(R_{j-1}^{-1}),\,h\right\rangle +\left\langle \phi_{1}(\beta_{j}^{-1}R_{j}^{-1}),\,h\right\rangle =-\left\langle \phi_{1}\left(\sharp{\textstyle \frac{\partial R}{\partial\alpha_{j}}}\right),\,\phi_{2}\left(\alpha_{j}\right)\right\rangle 
\end{align*}
where we also used the cocycle property $\phi_{1}\left(\gamma\right)=-\Ad_{\rho}\left(\gamma\right)\cdot\phi_{1}\left(\gamma^{-1}\right)=-\gamma\cdot\phi_{1}\left(\gamma^{-1}\right)$.
An analogous computation for the integrals $\int_{R_{j-1}\alpha_{j}y}^{R_{j-1}\alpha_{j}\beta_{j}y}$
and $\int_{R_{j-1}\alpha_{j}\beta_{j}\alpha_{j}^{-1}y}^{R_{j}y}$,
and a substitution into Equation \eqref{eq:inner product 2} provides
the desired formula.
\end{proof}
Theorem \ref{prop: Bilinear relations for differentials of pi1module}
may be called the bilinear relations for periods of $\Ad\left(E_{\rho}\right)$
since it reduces to the classical Riemann's bilinear relations in
the one dimensional case, that is, when $\rho\in\hom(\pi_{1},\mathbb{C}^{*})$
(see \cite{Florentino}).

\subsection{Derivative at unitary representations}

From Theorem \ref{thm:kernel of Schottky map} and (\ref{eq:goodtangentspaces}),
we know that the kernel of the derivative of the Schottky map at a
good Schottky representation $\rho\in\Hom\left(\pi_{1},\,G\right)$
is given by 
\[
\mathrm{\ker d_{[\rho]}\mathbf{V}\,\cong\,}T_{[\rho]}\mathbb{S}\bigcap\mathrm{Im}\,d_{0}Q_{\rho}\,\cong\,\left(\mathfrak{z}^{g}\oplus H^{1}\left(F_{g},\,\mathfrak{g}_{\Ad\rho_{2}}\right)\right)\bigcap\mathrm{Im}d_{0}Q_{\rho}
\]
where $\rho=\left(\rho_{1},\,\rho_{2}\right):F_{g}\to Z\times G$,
as in Section 2. According to Proposition \ref{PeriodMapEQdiFQrho},
since $d_{0}Q_{\rho}$, coincides with $P_{\Ad_{\rho}}$, we can write
the kernel as the following intersection 
\[
\mathrm{\ker d_{[\rho]}\mathbf{V}\cong}\left(\mathfrak{z}^{g}\oplus H^{1}\left(F_{g},\,\mathfrak{g}_{\Ad\rho_{2}}\right)\right)\bigcap\mathrm{Im}P_{\Ad_{\rho}}.
\]

Note that we are identifying the cohomology space, given by $\mathfrak{z}^{g}\oplus H^{1}\left(F_{g},\,\mathfrak{g}_{\Ad\rho_{2}}\right),$
with its image under the natural inclusion $\mathfrak{z}^{g}\oplus H^{1}\left(F_{g},\,\mathfrak{g}_{\Ad_{\rho_{2}}}\right)\subset H^{1}\left(\pi_{1},\,\mathfrak{g}_{\Ad_{\rho}}\right).$

In the case $\rho$ is strict, $T_{[\rho]}\mathbb{S}_{s}\cong H^{1}\left(F_{g},\,\mathfrak{g}_{\Ad\rho_{2}}\right)$
and we can identify the cohomology space $H^{1}\left(F_{g},\,\mathfrak{g}_{\Ad_{\rho_{2}}}\right)$
with its image under the natural inclusion $H^{1}\left(F_{g},\,\mathfrak{g}_{\Ad_{\rho_{2}}}\right)\subset H^{1}\left(\pi_{1},\,\mathfrak{g}_{\Ad_{\rho}}\right).$ 
\begin{lem}
\label{prop:G ss Period map omega zero implies omega zero} Let $\rho$
be a unitary and good strict Schottky representation. Consider $\omega\in H^{0}\left(X,\Ad\left(E_{\rho}\right)\otimes\Omega_{X}^{1}\right)$
such that $P_{\Ad_{\rho}}\left(\omega\right)\in H^{1}\left(F_{g},\mathfrak{g}_{\Ad_{\rho_{2}}}\right)$
(in particular, the component of $P_{\Ad_{\rho}}\left(\omega\right)$
in $\mathfrak{z}^{g}$ vanishes). Then $\omega=0$. In other words,
under the stated conditions: 
\[
H^{1}\left(F_{g},\,\mathfrak{g}_{\Ad_{\rho_{2}}}\right)\bigcap\mathrm{Im}P_{Ad_{\rho}}=0.
\]
\end{lem}
\begin{proof}
According to Theorem \ref{prop: Bilinear relations for differentials of pi1module},
the hermitian inner product of $\omega$ verifies \linebreak{}
 $\left(\omega,\,\omega\right)=\left\langle \!\left\langle P_{\Ad_{\rho}}\left(\omega\right),\,P_{\Ad_{\rho}}\left(\omega\right)\right\rangle \!\right\rangle .$
In this case the cup product of this class with itself is $P_{Ad_{\rho}}\left(\omega\right)\cup P_{Ad_{\rho}}\left(\omega\right)\in H^{2}\left(F_{g},\mathfrak{g}_{\Ad_{\rho_{2}}}\right).$
Since for a free group $F_{g}$, $H^{2}\left(F_{g},\,\mathfrak{g}_{\Ad_{\rho_{2}}}\right)=0$,
we obtain $P_{Ad_{\rho}}\left(\omega\right)\cup P_{Ad_{\rho}}\left(\omega\right)=0$
and by Theorem \ref{prop: Bilinear relations for differentials of pi1module},
$\omega=0$ since the Hermitian product is non-degenerate. 
\end{proof}
We can now prove our main result of this section (Theorem B of the
introduction). Let $\mathbf{V}_{s}:\mathbb{S}_{s}\rightarrow\mathcal{M}_{G}$
be the restriction of the Schottky moduli map (Definition \ref{def:Schottky-moduli-map})
to strict Schottky space. 
\begin{thm}
\label{thm: main thm}Let $\rho$ be a good and unitary Schottky representation,
and suppose that $[E_{\rho}]\in\mathcal{M}_{G}$ is a smooth point.
If $\rho$ is strict, the derivative of the Schottky moduli map, $d_{[\rho]}\mathbf{V}_{s}:T_{\left[\rho\right]}\mathbb{S}_{s}\rightarrow T_{\left[E_{\rho}\right]}\mathcal{M}_{G}$,
is an isomorphism. In the general case, the derivative of the Schottky
moduli map $\mathbf{V}:\mathbb{S}^{*}\to\mathcal{M}_{G}$ has maximal
rank at $[\rho]$. In particular, $\mathbf{V}$ is a local submersion
so that, locally around $[\rho]$, it is a projection with $\dim\left(\mathbf{V}^{-1}\left(\left[E_{\rho}\right]\right)\right)=g\dim Z^{\circ}$. \end{thm}
\begin{proof}
In the case $\rho$ is strict, 
\[
\ker d_{[\rho]}\mathbf{V}_{s}\cong H^{1}\left(F_{g},\,\mathfrak{g}_{\Ad_{\rho_{2}}}\right)\bigcap\mathrm{Im}\left(P_{\Ad_{\rho}}\right)
\]
and by Lemma \ref{prop:G ss Period map omega zero implies omega zero},
$\dim\ker d_{[\rho]}\mathbf{V}_{s}=0.$ Since, by Theorem \ref{Thm strict Schotky fst coh grp},
$\dim T_{\left[\rho\right]}\mathbb{S}_{s}=\left(g-1\right)\dim G+\dim Z$
and by \cite[Theorem 5.9]{Ramanathan2}, $\dim\mathcal{M}_{G}=\left(g-1\right)\dim G+\dim Z$,
thus 
\[
\dim T_{\left[\rho\right]}\mathbb{S}_{s}=\dim\mathcal{M}_{G},
\]
and we conclude that $d_{[\rho]}\mathbf{V}_{s}$ is an isomorphism
at $[\rho]$, where $\rho$ is a good and unitary strict Schottky
representation.

In the general case, by (\ref{eq:goodtangentspaces}), we have $T_{[\rho]}\mathbb{S}\cong\mathfrak{z}^{g}\oplus T_{[\rho_{2}]}\mathbb{S}_{s}$,
where $\rho_{2}$ is a good and unitary strict Schottky representation.
The tangent space $T_{[\rho_{2}]}\mathbb{S}_{s}$ can be identified
with a subspace of $T_{[\rho]}\mathbb{S}$. By the previous case,
$d_{[\rho_{2}]}\mathbf{V}$ is an isomorphism, so if we take as domain
$T_{[\rho]}\mathbb{S}$, $d_{[\rho]}\mathbf{V}$ remains surjective
with $\dim\ker d_{[\rho]}\mathbf{V}=g\dim Z^{\circ}$, because by
Corollary \ref{cor: dimension and iso of smooth Schottky rep } $\dim T_{[\rho]}\mathbb{S}=\dim T_{[\rho_{2}]}\mathbb{S}_{s}+g\dim Z^{\circ}$. \end{proof}
\begin{rem}
If $\rho$ is a unitary representation of $\hom(\pi_{1},G)$, the
corresponding $G$-bundle is semistable, by the main result in Ramanathan
\cite{Ramanathan1}. Assuming that $g\geq3$ and $\rho$ is good and
unitary, then $[E_{\rho}]$ is stable and smooth in $\mathcal{M}_{G}$,
by Biswas-Hoffmann \cite[Lemma 2.2]{BiswasHofmannTorelli}. 
\end{rem}
In the case $G$ is semisimple, the previous theorem implies the following. 
\begin{cor}
\label{cor:maintheorem} Let $G$ be semisimple. Then, at a good and
unitary Schottky representation $\rho$, the derivative of the Schottky
map, $d_{[\rho]}\mathbf{V}\,:\,T_{\left[\rho\right]}\mathbb{S}\rightarrow T_{\left[E_{\rho}\right]}\mathcal{M}_{G}$,
is an isomorphism. \end{cor}
\begin{proof}
First of all notice that the dimension of both spaces is the same.
Indeed, since $G$ is semisimple, $\dim Z=0$. Moreover, applying
Corollary \ref{cor: dimension and iso of smooth Schottky rep } to
$T_{\left[\rho\right]}\mathbb{S}$ we get $\dim T_{\left[\rho\right]}\mathbb{S}=\left(g-1\right)\dim G$
and by \cite[Theorem 5.9]{Ramanathan2}, $\dim\mathcal{M}_{G}=\left(g-1\right)\dim G$.
By Theorem \ref{thm: main thm}, $\ker d_{[\rho]}\mathbf{V}=0$, so
the result follows. 
\end{proof}

\section{Some Special Classes of Schottky Bundles\label{chap:Particular-cases}}

In this section, we consider two special classes of Schottky $G$-bundles
over a compact Riemann surface $X$: the case when $G$ is a connected
algebraic torus (over a general $X$); and general $G$-bundles over
an elliptic curve ($X$ has genus 1). Recall that, by slight abuse
of terminology, we say that a bundle is flat if it admits a holomorphic
flat connection.

\subsection{Abelian Schottky $G$-bundles}

Let $G$ be a complex connected algebraic torus. Then, it is well
known that $G$ is isomorphic to $(\mathbb{C}^{*})^{n}$, for some
$n\in\mathbb{N}$. So, we fix $G=(\mathbb{C}^{*})^{n}$, and note
that, in this situation, Schottky spaces are smooth varieties for
any $g$. Indeed, the space of strict Schottky representations becomes
\[
\mathbb{S}_{s}=\Hom\left(F_{g},\,(\mathbb{C}^{*})^{n}\right)\cong(\mathbb{C}^{*})^{ng}
\]
and $\mathbb{S}=\Hom\left(F_{g},\,(\mathbb{C}^{*})^{n}\times(\mathbb{C}^{*})^{n}\right)\cong(\mathbb{C}^{*})^{2ng}$.
We now generalize the result of \cite{Florentino}, stating that all
flat line bundles are strict Schottky $\mathbb{C}^{*}$-bundles. 
\begin{prop}
\label{prop: C star}Let $E$ be a $(\mathbb{C}^{*})^{n}$-bundle
over a compact Riemann surface $X$. Then $E$ is a \emph{strict}
Schottky bundle if and only if it is flat. \end{prop}
\begin{proof}
If $E$ is Schottky then it is induced by a representation, so $E$
is flat, by definition. Assume now that $E$ is a flat $G$-bundle,
with $G=(\mathbb{C}^{*})^{n}$. As in Proposition \ref{prop:product-bundles},
we can view $E$ as an ordered $n$-tuple of $\mathbb{C}^{*}$-bundles
$(E_{1},\cdots,E_{n})$, and then each $E_{i}$ admits a flat connection.
On the other hand, it is well known that $\mathbb{C}^{*}$-bundles
are equivalent to line bundles, i.e., vector bundles of rank one.
So, each $E_{i}$ is a line bundle of degree zero (since $E_{i}$
is flat). According to \cite{Florentino}, every line bundle with
degree $0$ is a Schottky vector bundle, that is, a \emph{strict}
Schottky $\mathbb{C}^{*}$-bundle. So, this implies that $E_{i}$
is strict Schottky for every $i=1,\cdots,n$. Hence, by Proposition
\ref{prop:product-bundles}, $E$ is also a strict Schottky bundle. \end{proof}
\begin{rem}
(1) Replacing $(\mathbb{C}^{*})^{n}$ by an arbitrary reductive abelian
group $G$, not necessarily connected, one can show that the previous
result is still valid. \\
 (2) It has been shown in \cite{FlorentinoLudsteck} that unipotent
bundles (arising from successive extensions of $\mathbb{C}^{*}$-bundles)
are also Schottky, and in fact, there is an equivalence of categories
between flat unipotent bundles over $X$, and unipotent representations
of free groups. 
\end{rem}
For $G=(\mathbb{C}^{*})^{n}$, it is well known that all $G$-bundles,
considered as (ordered) $n$-tuples of line bundles, are semistable.
Thus, the moduli space of semistable $(\mathbb{C}^{*})^{n}$-bundles
coincides with the space of all $(\mathbb{C}^{*})^{n}$-bundles: 
\[
\mathcal{M}_{(\mathbb{C}^{*})^{n}}\cong H^{1}(X,\,(\mathcal{O}_{X}^{*})^{n})\cong H^{1}(X,\mathcal{O}_{X}^{*})^{n}.
\]
It is well known that this sits in an exact sequence 
\[
H^{1}(X,\mathcal{O}_{X})^{n}\to H^{1}(X,\mathcal{O}_{X}^{*})^{n}\to\mathbb{Z}^{n},
\]
whose last morphism is the multi-degree, or first Chern class. So,
the space of flat $(\mathbb{C}^{*})^{n}$-bundles coincides with the
kernel of the degree map, that is, with 
\[
(H^{1}(X,\mathcal{O}_{X}^{*})^{n})^{0}\cong J(X)^{n},
\]
where $J(X)$ is the Jacobian of $X$. In this context the strict
Schottky moduli map looks as follows 
\[
\mathbf{V}_{s}:\Hom\left(F_{g},\left(\mathbb{C}^{*}\right)^{n}\right)\to J(X)^{n},
\]
and Proposition \ref{prop: C star} implies that $\mathbf{V}_{s}$
is onto (then, of course $\mathbf{V}:\mathbb{S}\to J(X)^{n}$). Also
note that $\dim\mathbb{S}_{s}=\dim J\left(X\right)^{n}=ng$. So this
description reproduces the line bundle case, for $n=1$, treated in
\cite{Florentino}.

\subsection{Schottky $G$-bundles over elliptic curves\label{sec:Semistability over elliptic curves}}

{} In this section, we consider principal Schottky bundles over an elliptic
curve $X$, the case $g=1$, which was excluded in previous sections\footnote{Note that the case $X=\mathbb{P}^{1}$ ($g=0$) is irrelevant, as
$\pi_{1}$ is trivial and so are Schottky representations.}. Firstly, we consider the case of vector bundles over an elliptic
curve and recall some results relating flat connections, semistability
and the Schottky property. Then, we relate $G$-bundles with the corresponding
adjoint bundle in order to translate some of the previous properties
to this case.

We begin by recalling the following theorem, due to Atiyah and Tu
\cite{Atiyah,Tu}, which relates semistability with the indecomposable
property. 
\begin{thm}
\cite{Tu}\label{AtiyahTu thm} Every indecomposable vector bundle
over an elliptic curve is semistable; it is stable if and only if
its rank and degree are relatively prime. 
\end{thm}
To relate flatness with semistability we now use Weil's theorem \cite[Theorem 10]{Weil},
which states that a vector bundle is flat if and only all its indecomposable
components have degree zero. 
\begin{prop}
\label{flat semistable} Let $V$ be a vector bundle over an elliptic
curve $X$. Then, $V$ is flat if and only if $V$ is semistable of
degree zero. \end{prop}
\begin{proof}
By the Krull-Remak-Schmidt Theorem, we can write $V$ as a direct
sum of indecomposable subbundles 
\[
V=\oplus_{i=1}^{n}V_{i}.
\]
Suppose that $V$ is flat. By Weil's theorem mentioned above, $\deg(V_{i})=0$
and, by Theorem \ref{AtiyahTu thm}, each one of $V_{i}$'s are semistable.
Since the sum of semistable vector bundles of the same slope ($\mu(V_{i})=\deg(V_{i})/\mbox{rk}(V_{i})=0$)
is semistable (of the same slope), $V$ is semistable with $\deg(V)=0$.
Conversely, let $V$ be semistable of degree 0. Then $0=\deg(V)=\sum\deg(V_{i})$,
and if some $V_{i}$ has degree $\deg(V_{i})\neq0$ then, at least,
there is one $V_{j}$ with $\deg(V_{j})>0=\deg(V)$. By definition,
this contradicts the hypothesis that $V$ is semistable. Therefore,
all of these $V_{i}$'s have degree zero which implies, by \cite[Theorem 10]{Weil},
that every summand $V_{i}$ is flat. Since a direct sum of flat bundles
admits a natural flat connection, $V$ is itself flat. 
\end{proof}
In \cite[Thm. 6]{Florentino}, it is shown that all flat vector bundles
over elliptic curves are Schottky. By considering adjoint bundles,
we now establish similar conclusions for $G$-bundles over elliptic
curves. 
\begin{prop}
\label{propssflat} Let $X$ be an elliptic curve, $G$ a connected
reductive algebraic group and $E$ a $G$-bundle over $X$. Then the
following are equivalent:
\begin{enumerate}
\item $E$ is semistable; 
\item $\mathrm{Ad}(E)$ is semistable; 
\item $\mathrm{Ad}(E)$ is flat. 
\end{enumerate}
If $G$ is semisimple, then all conditions above are equivalent to:

\emph{$\ $(4)} $E$ is flat. \end{prop}
\begin{proof}
\cite[Proposition 2.10]{AnchoucheBiswas} states that $E$ is semistable
if and only if $\mathrm{Ad}\left(E\right)$ is also semistable; thus
we obtain the equivalence between the two first assertions. The statements
(2) and (3) are equivalent by Proposition \ref{flat semistable}.
Finally, we can use \cite[Proposition 2.2]{AzadBiswas}, in the case
that $G$ is semisimple, to conclude that $E$ admits a flat connection
if and only if $\mathrm{Ad}(E)$ admits one (see also \cite{BaranovskyGinzburg}). \end{proof}
\begin{rem}
When $G$ is reductive, although the equivalence $(3)\Leftrightarrow(4)$
is not generally valid, we still can say that if $E$ is flat, then
$Ad(E)$ is flat (see \cite[Proposition 3.1]{AzadBiswas}). \end{rem}
\begin{thm}
\label{flat then schottky principal} Let $X$ be an elliptic curve,
and let $E$ be a $G$-bundle over $X$, for a connected reductive
algebraic group $G$. Then, $E$ is flat if and only if $E$ is Schottky.
In other words, for $g=1$, the Schottky uniformization map $\mathbf{W}:\mathbb{S}\to M_{G}$
is surjective. \end{thm}
\begin{proof}
A Schottky $G$-bundle $E$ is, by definition, flat. If the $G$-bundle
$E$ admits a flat connection then it induces a flat connection in
$\mathrm{Ad}\left(E\right)$. Using \cite[Theorem 6]{Florentino},
$\mathrm{Ad}\left(E\right)$ is strict Schottky, because it is a flat
vector bundle of degree $0$. By Proposition \ref{AdEToESchottky},
since $\Ad\left(E\right)$ is Schottky and $E$ is flat we obtain
that $E$ is a Schottky $G$-bundle. \end{proof}
\begin{rem}
When $G$ has a connected center the above result, together with Proposition
\ref{connected-center}, implies that, on an elliptic curve, $E$
is flat if and only if it is a \emph{strict} Schottky $G$-bundle. 
\end{rem}
The following Corollary follows directly from Proposition \ref{propssflat}
and Theorem \ref{flat then schottky principal}. 
\begin{cor}
\label{cor: all ss are Schottky}Let $X$ be an elliptic curve and
let $G$ be a semisimple algebraic group. Then every semistable $G$-bundle
over $X$ is Schottky and it is strict Schottky if $Z$ is connected.
In particular, the Schottky moduli map $\mathbf{V}:\mathbb{S}^{*}\to\mathcal{M}_{G}$
is surjective. \end{cor}
\begin{rem}
(1) A statement that includes both cases in Sections 9.1 and 9.2 is
the following: Let $X$ be a compact Riemann surface, $G$ a connected
reductive group, and $E$ a $G$-bundle on $X$. If either $\pi_{1}$
or $G$ are abelian, then $E$ is flat if and only if $E$ is Schottky.\\
(2) In the case $g=1$, since $\pi_{1}\cong\mathbb{Z}^{2}$, there
are no irreducible representations (nor good representations) $\rho:\pi_{1}\to G$,
for non-abelian $G$. However, the moduli space of semistable $G$-bundles
is non-empty, and is generally a weighted projective space (see for
example \cite{FriedmanMorganWitten}).
\end{rem}
Schottky vector bundles over elliptic curves, have been applied to
an analytic construction of non-abelian theta functions for $G=SL_{n}\mathbb{C}$,
which is completely analogous to the abelian classic case, \cite{FlorentinoMouraoNunes03,FlorentinoMouraoNunes04},
in the context of geometric quantization of the moduli space of vector
bundles. In a future work, we plan to give a generalization of these
results to Schottky $G$-bundles over an elliptic curve, for a general
reductive algebraic group $G$. 

\bibliographystyle{amsalpha}
\bibliography{Bibliography}

\end{document}